\documentclass{amsart}[12 pt]
\usepackage{amsmath, amsfonts, amssymb, amsthm, euscript, amscd, latexsym, mathrsfs, bm, bbm, color, xcolor, mathtools}
 \usepackage{titlesec}
 \usepackage[pdftex,colorlinks=true,
                       pdfstartview=FitV,
                       linkcolor=blue,
                       citecolor=blue,
                       urlcolor=blue,
           ]{hyperref}

\def\aa#1{ \begin{align*} #1 \end{align*} }
\def\aaa#1{ \begin{align} #1 \end{align} }
\def\mm#1{ \begin{multline*} #1 \end{multline*} }
\def\mmm#1{ \begin{multline} #1 \end{multline} }

\def\een#1{ \begin{equation}\begin{array}{l} #1 \end{array} \end{equation}}

\newtheorem{thm}{\sc Theorem}[section]
\newtheorem{lem}{\sc Lemma}[section]

\newtheorem{rem}{\sc Remark}[section]

\newcommand{\eps}{\varepsilon}
\newcommand{\pl}{\partial}
\newcommand{\gt}{\geqslant}
\newcommand{\lt}{\leqslant}
\newcommand{\te}{\theta}
\newcommand{\sub}{\subset}
\newcommand{\dl}{\delta}
\newcommand{\al}{\alpha}
\newcommand{\gm}{\gamma}
 
 \newcommand{\Dl}{\Delta}
 \newcommand{\la}{\lambda}
 
 \newcommand{\sg}{\sigma}

\newcommand{\om}{\omega}
\newcommand{\mc}{\mathcal}
\newcommand{\Om}{\Omega}

\newcommand{\C}{{\rm C}}

\newcommand{\td}{\tilde}

\newcommand{\E}{\mathbb E}
\newcommand{\we}{\wedge}

\newcommand{\x}{\times}
\newcommand{\mto}{\mapsto}
\newcommand{\PP}{\mathbb P}

\newcommand{\cov}{{\rm cov}\,}
\newcommand{\tr}{{\rm tr}}

\newcommand{\rf}{\eqref}
\newcommand{\bi}{\begin{itemize}}
\newcommand{\ei}{\end{itemize}}

\DeclareMathOperator{\ind}{\mathbbm{1}}
\newcommand{\lap}{\Delta}
\newcommand{\nab}{\nabla}
\newcommand{\lb}{\label}
\newcommand{\fdot}{\,\cdot\,}
\newcommand{\hsp}{\hspace{-2mm}}
\newcommand{\hs}{\hspace{-1mm}}

\newcommand{\sig}{\varsigma}
\newcommand{\tet}{\vartheta}

\def\Rnu{{\mathbb R}}

\def\Nnu{{\mathbb N}}

\def\ffi{\varphi}

\def\com#1{}

\long\def\symbolfootnote[#1]#2{\begingroup%
\def\thefootnote{\fnsymbol{footnote}}\footnote[#1]{#2}\endgroup}
\titleformat{\section}[hang]{\large\bfseries}{\thesection.}{1ex}{}{}
\titleformat{\subsection}[hang]{\normalsize\bfseries}{\thesubsection}{2ex}{}{}
\titleformat{\subsubsection}[hang]{\small\bfseries}{\thesubsubsection}{2ex}{}{}

\allowdisplaybreaks

\include{srctex.sty}

\usepackage{lipsum}
\usepackage{algorithmic}

\title[Multidimensional stochastic Burgers equation via FBSDEs]{Classical solution to a multidimensional stochastic Burgers equation via forward-backward SDEs}

\author{Alberto Ohashi and Evelina Shamarova}
\address{Departamento de Matem\'atica, Universidade Federal da Para\'iba, 13560-970, Jo\~ao Pessoa - Para\'iba, Brazil}
\email{ohashi@mat.ufpb.br} 
\email{evelina@mat.ufpb.br}

\begin{document}

\maketitle

\vspace{-7mm}

\begin{abstract}
In this paper, we address the problem of 
 existence and uniqueness of a global classical
 solution to a multidimensional stochastic Burgers equation 
 without gradient-type assumptions on the force  or the initial condition.
 The equation is first transformed to a random PDE, and then  solved 
via the associated forward-backward SDE. 
 Additionally, we obtain a new a priori gradient estimate valid for a large class of second-order quasilinear parabolic PDEs
 which becomes an important tool in our approach.
Also, we study the stochastic Burgers equation in the vanishing viscosity limit.
\end{abstract}

\vspace{3mm}

{\footnotesize
{\bf Keywords:} Stochastic Burgers equation, Forward-backward SDEs, Gradient estimate, Vanishing viscosity limit.

\vspace{3mm}

{\bf AMS subject classifications:} 60H15, 60H10, 35Q35, 35K59}

\section{Introduction}
In this article, we obtain the existence and uniqueness of a global classical solution to the multidimensional  stochastic Burgers equation
 \aaa{
\label{viscous}
y(t,x) =  h(x) + \int_0^t \hs \big[\nu \Dl y(s,x) - (y,\nabla)y(s,x)  + f(s,x,y) \big] ds + \eta(t,x)
 }
on $[0,T] \x \Rnu^n$, 
where $h$ is a random initial data, $f$ is a deterministic function representing  force, and $\eta(t,x)$ is a noise smooth in $x$
and rough in time. In particular, $\eta(t,x)$ can be a stochastic integral $\int_0^t g(s,x) dB_s$, assumed to be 
defined for each $x$, but this choice does not affect our analysis.
Importantly, we do not assume that any of the functions $f$, $\eta$, or $h$ are of gradient form.

In the past two decades many works have been dedicated to the problem of Burgers turbulence (see, e.g.,
 \cite{assing, bec1, bec2, bertini,  boritchev, Christian, daprato, gotoh, gurbatov1}), that is, the study of solutions to 
 a Burgers equation with a random initial condition or force.  In the extensive survey on Burgers turbulence \cite{bec}, Bec and Khanin 
refer the multidimensional extension of a  stochastic Burgers equation in the non-potential case as an important open question.  
The authors illustrate that
 when the forcing and the initial data are potential (i.e., represented as gradients of other functions), 
 the potential character of the velocity field is conserved by the dynamics, so the situation carry many similarities 
 with the one-dimensional case \cite{bec}. Further, the authors in \cite{bec} explicitly  pose 
  the question  of what happens when the potentiality assumption of the flow is dropped.

  Our main motivation in studying the multidimensional 
  viscous Burgers equation with smooth random forces is its application to the theory of hydrodynamical turbulence
  \cite{bec, boritchev, iturriaga1}.
As such, equations of form  \rf{viscous}  are frequently used 
 as a model of randomly driven Navier-Stokes equations without pressure \cite{14,12}.
 
 In this work, we propose a method of obtaining a global classical solution to  stochastic Burgers equation \rf{viscous} based on a fixed
 point argument of the associated forward-backward SDE (FBSDE) and a gradient estimate.  
 First, we transform \rf{viscous} to a random PDE, and then introduce
 a sequence of stopping times making the noise globally bounded. This
 allows us  to apply FBSDE techniques similar to the case of deterministic PDEs \cite{delarue, peng92}, and also, 
 to make use of our own result on a gradient estimate for PDEs by means of FBSDEs.




The  interest in Burgers turbulence is motivated by its applications in cosmology \cite{zeldovich}, fluid dynamics \cite{15}, superconductors \cite{blatter}, etc.
It is known that the  Burgers equation arises as an asymptotic form of various nonlinear dissipative systems \cite{bec}.
That is why
a one-dimensional stochastic Burgers equation has been intensely studied over the last two decades in a variety of contexts and based on different techniques. The literature is vast, so we refer the reader to the series of works \cite{bertini,daprato,daprato2, gyongy}, and references therein. 
The stochastic multidimensional  potential case, i.e., when the force and the initial data are of the gradient form, has also been studied by some authors
\cite{assing, boritchev, chen, Christian, iturriaga}. Since the potential Burgers equation can be reduced to a one-dimensional parabolic equation 
by a number of known approaches (see, e.g., \cite{boritchev, chen, Christian}), the analysis is significantly simplified.
We remark that in the present article, we consider the non-potential case for both, the random force and the initial condition, 
which does not allow us to apply any of the above techniques. 

  Further, we would like to mention article  \cite{brzezniak}, where
  the authors prove the existence and uniqueness of a global strong solution to a non-potential multidimensional stochastic Burgers equation 
  in the $L_p$-space with the number $p$ bigger than the dimension of the equation.
  Although the stochastic Burgers equation in \cite{brzezniak}  has the form similar to \rf{viscous}, the approach of the aforementioned work
  completely differs from ours. Besides, 
  from the hydrodynamical turbulence point of view, 
 $L_p$-solutions do not appear suitable since they do not convey the meaning of the solution
to \rf{viscous} as the velocity of a fluid at a given point $x$ in the space \cite{landau}. 
Also, our noise term is not assumed to take any specific form, unlike \cite{brzezniak}. In fact, 
the choice of the forcing term $\eta(t,x)$ 
in physics literature is frequently made on the basis of the covariance of the form
$\cov\!(\dot \eta^i(t,x),\dot\eta^j(t',x')) = \dl(t-t') \ffi_{ij}(x-x')$ (see, e.g., \cite{14,12}). However, the above relation is
not satisfied by the stochastic-integral-type noise. 
Remark that in \cite{brzezniak}, the choice of the noise term as a stochastic integral plays a crucial role
 in the analysis.
Another advantage of our method is the use of the associated FBSDE, which may allow the results of paper \cite{delarue3} 
 on a forward-backward stochastic algorithm for PDEs
to be applied to tackle equation \rf{viscous}  numerically. 

 Furthermore, we mention that in the deterministic case, the global existence and uniqueness of a classical solution to 
the multidimensional Burgers equation
   is known due to the results of Ladyzhenskaya 
  et al \cite{lady}, and follows as a particular case of a more general theory for systems of quasilinear parabolic PDEs. 
  However, the results of \cite{lady} are not applicable to equation \rf{viscous} since the noise is not differentiable in time.

  As a byproduct of our approach, we obtain
 an a priori gradient estimate valid for a large class of quasilinear second order parabolic PDEs.
Our bound is obtained exclusively  by using the associated FBSDE. 
  Previously, a gradient estimate by means of FBSDE techniques was obtained in \cite{delarue2}. 
  However, the result of \cite{delarue2} cannot be applied to the present case. Indeed, in our work,  the gradient
estimate is used in the process of construction of the solution by glueing the solutions on short-time intervals, i.e., 
we deal with solutions defined on subintervals of $[0,T]$ but not on the entire interval. In this situation,
 the results of \cite{delarue2} do not guarantee that the gradient bound will be uniform 
 over the length of the subinterval, while our result does guarantee that. Thus, our gradient estimate
appears completely suitable for solving some class of PDEs by means of FBSDEs.
Additionally, our approach to obtaining this bound is significantly simpler and shorter than in \cite{delarue2},
although it is valid for a smaller class of PDEs.

  Also, we remark that the classical book on quasilinear parabolic PDEs
  by Ladyzhenskaya et al \cite{lady} only provides an a priori gradient estimate for an initial-boundary value problem
  on a bounded domain.

 Finally, we study  the vanishing viscosity limit of equation \rf{viscous}.
 We investigate this problem only locally. 
 Namely, we prove that on a small random time interval,
 there exists a unique classical solution to the inviscid stochastic Burgers equation and the solutions to viscous stochastic
 Burgers equations with the same force terms and the initial data converge to the inviscid solution uniformly in space and time. 
 Note that even on a short time interval, many authors investigated 
the vanishing viscosity limit in hydrodynamics problems.
 As such, Ebin and Marsden \cite{ebin} proved the convergence of local Sobolev-space-valued solutions of the Navier-Stokes equation
 to local solutions of the Euler equation. Golovkin \cite{golovkin} and Ladyzhenskaya \cite{lady2} obtained 
 the aforementioned convergence uniformly in space and time. Further, Ton \cite{ton} studied the local vanishing viscosity limit of a  multidimensional deterministic Burgers equation in an $L_2$-space. Furthermore, Brze\'zniak et al \cite{brzezniak} proved that
  viscous solutions to a potential stochastic Burgers equation
 converge locally to an inviscid viscosity solution.
It is known that even if the initial data and the force are smooth, a one-dimensional inviscid Burgers equation develops discontinuities (shocks)
 at a finite time, and, therefore, fails to have a global classical solution.
 Thus, one cannot expect a global uniform approximation of inviscid solutions by viscous. Finally, we remark that  the inviscid 
multidimensional  stochastic Burgers equations is also studied by means of  the associated stochastic forward-backward system.


\section{Existence and uniqueness of solution to equation \rf{viscous}}
\lb{global1}

In this section, we show that under assumptions (A1)--(A3) below, equation \rf{viscous} possesses a unique global  
solution $y(t,x)$ which is $\C^2$-smooth in $x$ and continuous in $t$.

\subsection{Assumptions and choice of the noise}
Let $(\Om,\mc F, \mc F_t, \PP)$ be a filtered probability space satisfying the usual conditions.

Assume the following:
\bi
\item[\bf (A1)] $f(t,x,y)$ is an $\Rnu^n$-valued deterministic function of class $\C^{0,2}_b([0,T]\x \Rnu^{2n})$.
 \item[\bf (A2)] $\eta(t,x)$ is an $\Rnu^n$-valued stochastic process which is 
 $\mc F_t$-adapted for each $x$;
 moreover, a.s., $\eta(t,x)$ is of class $\C^{0,4}_b([0,T] \x \Rnu^n)$ and 
 $\eta(0,x) = 0$.
  \item[\bf (A3)] For  each $x\in\Rnu^n$, $h(x)$ is an $\mc F_0$-measurable random variable, which, moreover, is 
 of class $\C^2_b(\Rnu^n)$ a.s. 
 \ei

Below, we give a few examples of the noise process $\eta(t,x)$ satisfying (A2). 

\textit{Example 1.} $\eta(t,x) = \int_0^t g(s,x) dB_s = \sum_{i=1}^d g_i(s,x)dB^i_s$, where 
$B^i_t$ are independent real-valued $\mc F_t$-Brownian motions, and the stochastic integral is defined for each $x\in\Rnu^n$.
Let us show that $\eta(t,x)$ verifies (A2) for some integrands $g(t,x)$. Namely, we assume:
\bi
\item[(i)] For each $x\in\Rnu^n$,  $g(t,x)$ is a progressively measurable stochastic process with values in
  $\Rnu^{d\x n}$ which takes the form $g(t,x) = \td g(t,\phi(x))$ for some $\Rnu^l$-valued
  random function $\phi(x)$ such that for each $x$ it is a
  random variable independent of $B_t$, $t\in [0,T]$.
 \item[(ii)]
  For each $t\in [0,T]$, $\td g(t,\fdot)$ is of class 
  $\C^{4+\al}_b(\Rnu^l)$ a.s., $\al \in (0,1)$;
   $\phi$ is of class $\C^{4}_b(\Rnu^n)$ a.s., and, furthermore, 
   $\E\int_0^T \|\td g(t,\fdot)\|^p_{\C^{4+\al}_b(\Rnu^l)} dt < \infty$
  for some  $p> 2+\te + (4+\te^2)^\frac12$,
 where $\te = \frac12\, \al^{-1}(n+1)$.
\ei
\begin{rem}
\lb{rem-b}
\rm
Recall that the space $C^{k+\al}_b(\Rnu^m)$,  $\al\in (0,1)$, $k\in\Nnu$,
is defined as the (Banach) space of functions $\zeta(x)$ possessing the finite norm
\aa{
\|\zeta\|_{C^{k+\al}_b(\Rnu^m)} = \|\zeta\|_{C^{k}_b(\Rnu^m)}
+ [\nab^k_x\zeta]^x_{\al},
}
where the H\"older constant $[\tet]^x_{\al}$ is defined as 
\aa{
[\tet]^x_{\al} = \sup_{\substack{x,x'\in \Rnu^m, \\ 0<|x-x'|<1}} \frac{|\tet(x) - \tet(x')|}{|x-x'|^{\al}}.
}
\end{rem}

\begin{rem}
\lb{rem2.2}
\rm
Assumptions (i) and (ii) are satisfied, in particular, 
when the functions $g(t,\fdot)$
 have a common compact support $D\sub\Rnu^n$. Then, take $\td g(t,x) = g(t,x)$ and $\phi(x) = x\xi(x)$, where $\xi(x)$ is
a $\C^\infty$-cutting function for $D$, i.e., $\xi(x) = 1$ if $x\in D$,
$\xi(x) = 0$ if $x$ is outside of $D_\dl$, a small $\dl$-neighborhood of $D$, and, moreover, $0\lt \xi(x)\lt 1$.
Furthermore, assume that $g(t,x)$ satisfies the regularity and integrability assumptions from (i) and (ii).
\end{rem}
\begin{lem}
\lb{lem99}
Under assumptions (i) and (ii), there is a version of 
the stochastic integral $\int_0^t g(s,x) dB_s$ which belongs to the space $\C^{0,4}_b([0,T] \x \Rnu^n)$.
\end{lem}
For the proof of Lemma \ref{lem99}, we need the next lemma.
\begin{lem}
\lb{h99}
Assume that for each $x\in\Rnu^n$, $\zeta(t,x)$  is a progressively measurable 
$\Rnu^{d\x n}$-valued stochastic process such that
for each $t\in [0,T]$, $\zeta(t,x)$
 belongs to class $\C^{1+\al}(\Rnu^n)$ and  $\E\int_0^T\|\zeta(s,\fdot)\|^p_{\C^{1+\al}(\Rnu^n)} ds<\infty$
for a number $p$ as in (ii). Then, 
the stochastic integral $\int_0^t \zeta(s,x) dB_s$ possesses a $\C^{0,1}([0,T]\x\Rnu^n)$-modification.
\end{lem}
\begin{proof}
Let, for any function $\tet(x)$, $\lap^k_\eps \tet(x) = \eps^{-1}\big(\tet(x+\eps e_k) - \tet(x)\big)$.
It is immediate to verify that
\mm{
\E\Big|\lap^k_\eps \int_0^t \zeta(s,x) dB_s - \lap^k_{\eps'} \int_0^{t'} 
\zeta(s,x') dB_s\Big|^p\\
\lt \gm(p,T) \, \E\int_0^T\|\zeta(s,\fdot)\|^p_{\C^{1+\al}(\Rnu^n)} ds\, \big(|\eps - \eps'|^{\al p} + |x-x'|^{\al p}
+|t-t'|^{\frac{p}2 -1}\big)
}
for some constant $\gm(p,T)$.
The statement of the lemma holds by the choice of $p$ (as in (ii)) and Kolmogorov's continuity theorem.
\end{proof}

\begin{proof}[Proof of Lemma \ref{lem99}]
Lemma \ref{h99} implies that the stochastic integral $\int_0^t \td g(s,z) ds$
possesses a $\C^{0,4}$-modification. This immediately implies that $\int_0^t \td g(s,\phi(x)) ds$ possesses
a $\C^{0,4}_b$-modification, i.e., its derivatives in $x$ are bounded.
\end{proof}

\textit{Example 2.}  Assume $g(t,\fdot)$ takes values in $\mc L(H, H^k(\Rnu^n))$, where 
$H$ is a Hilbert space and $H^k(\Rnu^n)$ is a Sobolev space with sufficiently large $k$. Further, let
$B_t$ be an $H$-valued cylindrical Brownian motion. Then,  
$\eta(t,\fdot) = \int_0^t g(s,\fdot) dB_s$ can be understood as an $H^k(\Rnu^n)$-valued stochastic integral. 
This implies that $\eta(t,x)$ is in $\C^{0,4}_b([0,T],\Rnu^n)$
by Kolmogorov's continuity theorem and Sobolev's imbedding $H^k(\Rnu^n) \hookrightarrow \C^4_b(\Rnu^n)$.

\textit{Example 3.}  Let $\dot W^i(t,x)$, $i=1,\ldots, n$, be independent space-time white noises, and let
$\dot W^i_\eps(t,x)$ be a regularization in $x$ of $\dot W^i(t,x)$, that is,
$\dot W^i_\eps(t,x) = (\dot W^i(t,\fdot) \ast \rho_{\eps})(x)$, where 
$\rho_{\eps}$ is a standard mollifier supported on the ball of radius $\eps$.
Alternatively, one can write 
$W^i_\eps(t,x) = (W^i(t,\fdot) \ast \pl^{n}_{x_1\ldots x_n}\rho_{\eps})(x)$, where $W^i(t,x)$  
is an $(n+1)$-parameter Brownian sheet. 
The filtration $\mc F_t$ can be taken as
follows $\sg\{W^i(s,x), 0\lt s\lt t,  i=1, \ldots, n, x\in\Rnu^n\}\vee \sg\{h(x), x\in\Rnu^n\} \vee \mc N$, where $\mc N$ is the collection of $\PP$-null sets. 
Remark that  $\cov(\dot W^i_\eps(t,x), \dot W^j_\eps(t',x'))=
\dl(t-t') \ffi_{ij}(x-x')$, where $\ffi_{ij}(y) = \dl_{ij} \int_{\Rnu^n} \rho_{\eps}(z)\rho_{\eps}(z+y) dz$.
Since we are interested in noises of class $\C^{0,4}_b(\Rnu^n)$, define $\dot \eta^i(t,x)$ as $\dot W^i_\eps(t,x) \xi(x)$,
where $\xi(x)$, $x\in\Rnu^n$, is a $\C^\infty$-cutting function for a bounded domain $D\sub\Rnu^n$ (see Remark \ref{rem2.2}).

\begin{rem}
\lb{om0}
\rm
Everywhere below, the set full $\PP$-measure, where  $\eta(t,x)$ and $h(x)$
belong to classes $\C^{0,4}_b([0,T] \x \Rnu^n)$ and $\C^2_b(\Rnu^n)$, respectively,
and $\eta(0,x)= 0$, will be denoted by $\Om_0$.
\end{rem}

\subsection{Local existence for stochastic Burgers-type equations}
\lb{local}
We start with the following lemma whose proof is straightforward.
\begin{lem}
\lb{lem14}
The substitution
\aaa{
\lb{substitution}
\hat y(t,x) = y(t, x) - \eta(t,x)
}
transforms \rf{viscous} to the following Burgers-type equation with random coefficients:
\aaa{
\lb{rforced}
\begin{cases}
\pl_t \hat y(t,x) =  \nu \Delta \hat y(t,z) - (\eta(t,x) + \hat y,\pl_x)\hat y(t,x) + F(t,x,\hat y),\\
\hat y(0,x) = h(x),
\end{cases}
}
where 
\aaa{
\lb{rforce}
 F(t,x,\hat y) = f(t,x,\hat y +  \eta(t,x) ) + \nu\lap \eta (t,x) - (\hat y + \eta,\pl_x )\eta (t,x).
}
\end{lem}

Everywhere below throughout this subsection, we assume that 
$\eta$, $F$, and $h$ possess deterministic bounds in the spaces $C^{0,2}_b([0,T]\x \Rnu^n)$, $C^{0,2}_b([0,T]\x \Rnu^{2n})$, and $\C^2_b(\Rnu^n)$, respectively.
Moreover, 
the force term $F$ is not assumed to necessarily take form \rf{rforce}.


In Theorem \ref{adapted} below, we prove 
the existence and uniqueness of a local $\mc F_t$-adapted $\C^{1,2}_b$-solution to \rf{rforced}. 
First, by doing the time change $\bar y(t,x) = \hat y(T-t,x)$, we transform \rf{rforced} to the backward equation
\aaa{
\lb{backward-pde}
\bar y(t,x) = h(x) + \int_t^T \big[ \nu \lap \bar y(s,x) - (\bar\eta(t,x) + \bar y, \nab)\bar y(s,x) + \bar F(s,x,y) \big] ds
}
with $\bar F(t,x,y) = F(T-t,x,y)$ and $\bar \eta(t,x) = \eta(T-t,x)$.

The following lemma will be useful.
\begin{lem}
\lb{lem2}
Let $W_t$ be a one-dimensional Brownian motion and $\mc B$  be a $\sg$-algebra independent of the (augmented) natural filtration $\mc F^W_t$ 
of $W_t$.
Assume that $\Phi_t$ is $\mc F^W_{t} \vee \mc B$-adapted and 
$\E\int_0^t |\Phi_s|^2 ds < \infty$, $t>0$. Then, 
$\E\Big[\int_{0}^{t} \Phi_s dW_s|\mc B\Big] = 0$ a.s.
\end{lem}
\begin{proof}
Let $0 = s_1 < \ldots < s_n = t$ be a partition.
Note that for a simple
$\mc F^W_{t} \vee \mc B$-adapted integrand $\Phi =\sum_i \Phi_{i}\ind_{[s_i,s_{i+1})}$, it holds that
\mm{
\E\Big[ \int_{0}^{t} \Phi_s dW_s|\, \mc B \Big]= \E\Big[\sum_i \Phi_{i}(W_{s_{i+1}}- W_{s_i})|\, \mc B\Big]\\ = \sum_i \E
\Big[ \Phi_{i} \, \E\big[(W_{s_{i+1} }- W_{s_i})|\mc F^{W}_{s_i}\vee \mc B\big] |\, \mc B\Big]= 0.
}
Further, we note that if a sequence $\{\Phi^{(n)}_t\}$ of
simple $\mc F^W_{t} \vee \mc B$-adapted integrands is such that
$\E \int_{0}^{t} (\Phi^{(n)}_s-\Phi_s)^2 ds  \to 0$,
then by the conditional Jensen's inequality
and It\^o's isometry,
$\E\big(\E\big[\int_{0}^{t} \big(\Phi^{(n)}_s -\Phi_s\big)dW_s|\, \mc B\big]\big)^2 \to 0$.
\end{proof}

Everywhere below, 
the symbol $\E_\tau$ will denote the conditional expectation with respect to $\mc F_{T-\tau}$.
\begin{thm}
\lb{adapted}
Let, the functions $\bar\eta(t,x)$, $\bar F(t,x,y)$, $h(x)$ satisfy the assumptions:
\bi
\item[1)] $\bar F(t,x,y)$ and $\bar\eta(t,x)$ are  
$\mc F_{T-t}$-adapted for each $x,y\in\Rnu^n$.
\item[2)] 
 $\bar \eta(t,x)$ and $h(x)$ a.s. belong to spaces
 $\C^{0,2}_b([0,T] \x\Rnu^n)$ and $\C^2_b(\Rnu^n)$, respectively,
 and possess a
deterministic bound $K$ with respect to the norms of the spaces.  
\item[3)] 
 $\bar F(t,x,y)$ is of class $\C^{0,2}([0,T] \x \Rnu^{2n})$ and satisfies the estimate
 $|\bar F(t,x,y)|+ |\nab_{(x,y)}\bar F(t,x,y)| + |\nab^2_{(x,y)}\bar F(t,x,y)| \lt K(1+|y|)$ a.s.
 \ei
Then,  there exists a constant $\gm_K$, depending only on $K$, such that
on $[T-\gm_K,T]$, there exists an $\mc F_{T-t}$-adapted $\C^{1,2}_b$-solution $\bar y(t,x)$ to equation  \rf{backward-pde}.
\end{thm}
\begin{proof}
In what follows, $\gm_i$, $\mu_i$, $i=1,2, \ldots$, are  positive deterministic constants that
may depend only on $p$ and $K$; in particular, they do not depend on $\nu$.
We will track the dependence of some constants on $\nu$ because it is important
for the next section.
Furthermore, the constants $\td \gm_{K}$, $\dot\gm_K$, $\hat \gm_{K}$, $\bar \gm_{K}$,  $\gm_{K}$  are 
positive and deterministic,
 that depend only on $K$; they determine the length of the interval.  Without loss of generality, these $\gm_K$-type constants are
 assumed to be smaller than $1$.

We prove the existence of an $\mc F_{T-t}$-adapted $C^{1,2}_b$-solution to \rf{backward-pde}
by means of the associated FBSDEs
(see \cite{delarue}, \cite{peng92}):
\aaa{
\lb{fbsde-new}
\begin{cases}
 X^{\tau,x}_t = x - \int_\tau^t \big( \bar\eta(s,X^{\tau,x}_s) + Y^{\tau,x}_s)\big) ds + \sqrt{2\nu} (W_t-W_\tau)\\
Y^{\tau,x}_t = h ( X^{\tau,x}_T) + \int_t^T  \bar F(s, X^{\tau,x}_s, Y^{\tau,x}_s) \, ds  - \int_t^T Z^{\tau,x}_s dW_s,
\end{cases}
}
where $W_t$ is an $n$-dimensional Brownian motion independent of the filtration $\mc F_{T-t}$,
and  the upper index $\tau,x$ means that the process $X^{\tau,x}_t$ starts at $x$ at time $\tau>0$.
For each $\tau\in (0,T)$, define the filtration
 \aaa{
 \lb{filtt}
 (\mc G^{\tau}_t)_{\tau\lt t \lt T}  =  \sg\{W_s-W_\tau, s \in [\tau,t]\}  \vee \mc F_{T-\tau}.
 }
In what follows, when it does not lead to misunderstanding, we will often skip the upper index ${\tau,x}$ 
in $(X^{\tau,x}_t,  Y^{\tau,x}_t,  Z^{\tau,x}_t)$ and similar processes to simplify notation.

\textit{Step1. 
Boundedness of $\E_\tau\, |Y^{\tau,x}_t|^p$ and modified FBSDE.}
Consider the backward SDE in \rf{fbsde-new}. 
From the assumptions of the theorem and It\^o's formula, it follows that $\E_\tau\, |Y^{\tau,x}_t|^p$ is bounded, a.s., 
for any 
solution $Y^{\tau,x}_t$ to this BSDE and for any  $\mc G^{\tau}_t$-adapted process $X^{\tau,x}_t$.
Indeed, since
\aa{
\bigl(|g|^p\bigr)'h  = p|g|^{p-2}(g,h); \;
\bigl(|g|^p\bigr)''h_1 h_2 = p(p-2)|g|^{p-4}(g,h_1)(g,h_2)+ p|g|^{p-2}(h_1,h_2)
}
for $p\gt 2$, then, a.s.,
\mmm{
\lb{ito-p}
\E_\tau\,|Y_t|^p
+ p(p-2)\int_t^T \E_\tau\, \bigl[|Y_s|^{p-4}\sum_{i=1}^n|(Z^i_s,Y_s)|^2\bigr]\,ds \\
+ p\int_t^T \hspace{-1mm} \E_\tau\,\bigl[|Y_s|^{p-2}|Z_s|^2\bigr]\,ds  = \E_\tau\,|h(X_T)|^p
+ 2p\int_t^T \hspace{-1mm}\E_\tau\,\bigl[|Y_s|^{p-2}(\bar F(s, X_s, Y_s), Y_s)\bigr]\, ds.
}
Since $|\bar F(t,x,y)| \lt K(1+|y|)$, then 
Young's inequality and Gronwall's lemma imply that  for every $(\tau,x)$,
\aaa{
\lb{y-bound}
\E_\tau\, |Y_t|^p \lt \gm_1 
\quad \text{and} \quad |Y^{\tau,x}_\tau| \lt (\gm_1)^\frac1{p} \quad \text{a.s.}
}
Moreover, $\gm_1$ is the same for all $(\tau,x) \in [0,T]\x\Rnu^n$.

Now let $\dl = (\gm_1)^\frac1{p}$ for some fixed $p$, and let
$\zeta_\dl(y) = \xi_\dl(y) y$, where $\xi_\dl(y)$ is a $\C^{\infty}$-cutting function for the ball $B_\dl$ of radius $\dl$ centered at the origin (see 
Remark \ref{rem2.2}).
We modify $\bar F$
 by introducing $\zeta_\dl(y)$ instead of $y$ as follows: 
\aaa{
\lb{gmforce}
\bar F_\dl(t,x,y)  = \bar F(t,x,\zeta_\dl(y)).
}
Together with Assumption 3), this implies that $|\bar F_\dl|$ is uniformly bounded by $K(1+\dl)$.
Further, consider the modified FBSDE
\aaa{
\lb{fbsde-modified}
\begin{cases}
 X^{\tau,x}_t = x - \int_\tau^t \big( \bar\eta(s,X^{\tau,x}_t) + Y^{\tau,x}_s)\big) ds + \sqrt{2\nu} (W_t-W_\tau)\\
Y^{\tau,x}_t = h ( X^{\tau,x}_T) + \int_t^T  \bar F_\dl(s, X^{\tau,x}_s, Y^{\tau,x}_s) \, ds  - \int_t^T Z^{\tau,x}_s dW_s.
\end{cases}
}
According to the results of \cite{delarue} (Theorem A.1), there exists a constant $\td \gm_K$,
depending only on $K$ (remark that $\dl$ also depends only on $K$), such that  
 whenever $T-\tau \lt \td\gm_{K}$, 
 system \rf{fbsde-modified} possesses a unique
  $\mc G^\tau_t$-adapted solution $(X^{\tau,x}_t, Y^{\tau,x}_t,Z^{\tau,x}_t)$ on $[\tau,T]$ such that $X^{\tau,x}_t$ and $Y^{\tau,x}_t$ have continuous paths a.s.

 {\it Step 2. Continuity of the map $(\tau,x)\mto Y^{\tau,x}_\tau$ and solution to the original FBSDE.}
    First, we prove that
the map $[T-\dot\gm_{K},T]\x \Rnu^n \to \C([T-\dot\gm_{K},T])$,   $(\tau,x) \mto (X^{\tau,x}, Y^{\tau,x})$
has  an a.s. continuous version for some constant $0<\dot\gm_K<\td\gm_K$.
 This continuity will  be
required, in particular, for the proof of differentiability of
$(X^{\tau,x}_t,Y^{\tau,x}_t)$ with respect to $x$.
Extend $X^{\tau,x}_s$ to $[T-\td\gm_K, \tau]$ by $x$, and $Y^{\tau,x}_s$ by $Y^{\tau,x}_\tau$.
By Corollary A.6 from \cite{delarue}, there exists a constant $\dot\gm_K<\td\gm_K$ such that  for any $x,x' \in\Rnu^n$, $\tau,\tau'\in [T-\dot\gm_K,T]$,
\mmm{
\lb{cont7}
\E\sup_{t\in [T-\dot\gm_K,T]} |X^{\tau,x}_t - X^{\tau',x'}_t|^{p} + \E\sup_{t\in [T-\dot\gm_K,T]} |Y^{\tau,x}_t - Y^{\tau',x'}_t|^{p} \\
\lt \gm_2 \big(|x-x'|^{p} + (1+|x|^{p}) |\tau-\tau'|^{\frac{p}{2}}\big),
}
where $p\gt 2$. Pick  $p>n$. Then, by Kolmogorov's continuity
criterion in Banach spaces (see, e.g., \cite{kunita}), there exists a continuous modification of the map
$[T-\dot\gm_K,T]\x \Rnu^n \to \C([T-\dot\gm_K,T])$,   $(\tau,x) \mto (X^{\tau,x}, Y^{\tau,x})$.
 In particular, the map $(\tau,x) \mto Y^{\tau,x}_\tau$ is continuous a.s. This and \rf{y-bound} imply that $\sup_{\tau,x}|Y^{\tau,x}_\tau| < \dl$ a.s.
 
 Further, according 
  to Corollary A.4 of \cite{delarue} and by the continuity in $(\tau,x)$ 
  obtained above, a.s.,
  \aaa{
  \lb{rel1}
  Y^{\tau,x}_t =  Y^{\tau,X^{\tau,x}_t}_\tau \quad \text{ for each} \;
\tau\in (T-\td\gm_K,T], \, t\in [\tau,T],\, x\in\Rnu^n.
  }
 Therefore,
  $(X^{\tau,x}_t, Y^{\tau,x}_t,Z^{\tau,x}_t)$ is also a solution to original FBSDE \rf{fbsde-new}  on $[\tau,T]$.

{\it Step 3. Differentiability of the FBSDEs solution in $x$. Boundedness of $\E_\tau\, |\pl X^{\tau,x}_t|^p$
and $\E_\tau\, |\pl Y^{\tau,x}_t|^p$.}
Now we proceed with the proof of differentiability. In Steps 3 and 4, we will write
$\bar F$ instead of $\bar F_\dl$ (defined by \rf{gmforce}) to simplify notation, and thus assuming
(without loss of generality)
that $\bar F$ is bounded together with its spatial derivatives up to the second order.

For any function $\al(x)$, define $\Dl^k_\eps \al(x) = \eps^{-1} \big(\al(x+\eps e_k) - \al(x)\big)$, $k=1, \ldots, n$,
where $\{e_k\}_{k=1}^n$ is the orthonormal basis in $\Rnu^n$.
In particular,
$\Dl^k_\eps X_t=\eps^{-1} (X_t^{\tau,x+\eps e_k} - X_s^{\tau,x})$, $k=1, \ldots, n$, and
$\Dl^k_\eps Y_t$, $\Dl^k_\eps Z_t$ are defined similarly.
Further, for a function $\Phi$ (which can be any of the functions $\bar F$, 
$h$, $\bar\eta$, or their gradients with respect to the spatial variables), 
we define $\nab_2 \Phi(t,u,v) = \pl_u \Phi(t,u,v)$, $\nab_3 \Phi(t,u,v)  = \pl_v \Phi(t,u,v)$. Furthermore, we define
\een{
\lb{f7}
\nab^{\eps,k}_2 \Phi_t = \int_0^1 \nab_2 \Phi(t, X_t+ \la\eps \Dl^k_\eps X_t,Y_t)d\la,\\
\nab^{\eps,k}_3 \Phi_t = \int_0^1 \nab_3 \Phi(t,  X_t, Y_t+ \la\eps \Dl^k_\eps Y_t)d\la,
}
and note that 
\aaa{
\lb{f7-}
\nab^{\eps,k}_2 \Phi_t =
 \int_0^1 \nab_2 \Phi(t, (1-\la) X_t^{\tau,x}+ \la X^{\tau,x+\eps e_k}_t,Y_t)d\la,
 }
 and similar for $\nab^{\eps,k}_3 \Phi_t$. 
 In case of just one spatial variable (like in $h$ or $\eta$), we write $\nab$ instead of $\nab_2$ and $\nab^{\eps,k}$ instead of $\nab^{\eps,k}_2$.
Note that
\aaa{
\lb{tayl}
\Dl^k_\eps \Phi_t  = \nab^{\eps,k}_2 \Phi_t \Dl^k_\eps X_t + \nab^{\eps,k}_3 \Phi_t \Dl^k_\eps Y_t.
}
It is immediate to verify that the triple $(\Dl^k_\eps X_t, \Dl^k_\eps Y_t, \Dl^k_\eps Z_t)$ solves the FBSDE
 \aaa{
   \lb{fbsde-diff-1}
\phantom{kkk}
 }
  \vspace{-12mm} 
\aa{
\hspace{7mm}
 \begin{cases}
 \Dl^k_\eps X_t = e_k - \int_\tau^t \big(\Dl^k_\eps Y_s + \nab^{\eps,k} \bar\eta_s \Dl^k_\eps X_s\big) \, ds,\\
 \Dl^k_\eps Y_t  = \nab^{\eps,k} h_T\, \Dl^k_\eps X_T
+ \int_t^T \big(\nab^{\eps,k}_2 \bar F_s \, \Dl^k_\eps X_s + \nab^{\eps,k}_3 \bar F_s \, \Dl^k_\eps Y_s\big) \, ds - \int_t^T   \Dl^k_\eps  Z_s \,dW_s       
\end{cases}
 }
on  the same time interval $[\tau,T]$, 
where we proved the existence and uniqueness of solution to \rf{fbsde-new}.
Additionally, we define $(\Dl^k_0 X_t, \Dl^k_0 Y_t, \Dl^k_0 Z_t)$ as the unique solution
to FBSDE \rf{fbsde-diff-1} whose coefficients are taken at $\eps = 0$.
Remark that setting $\eps = 0$ in \rf{f7-}, we obtain $\nab_2\Phi(t,X_t,Y_t)$ on the
right-hand side. 
  The existence and uniqueness of the triple $(\Dl^k_0 X_t, \Dl^k_0 Y_t, \Dl^k_0 Z_t)$
 follows from Theorem A.1 in \cite{delarue}.

Let us show that for $p\gt 2$, a.s., 
\aaa{
\lb{dif-bounds}
\max \big\{ \E_\tau |\Dl_\eps^k X_t|^p; \;   \E_\tau |\Dl_\eps^k Y_t|^p\big\} \lt   \gm_3 \quad \text{for all} \; \eps\gt 0, \, t\in [\tau,T].
}
It\^o's formula and the BSDE in
 \rf{fbsde-diff-1} imply 
 \mm{
\E_\tau|\Dl_\eps^k Y_t|^p
+  p(p-2)\int_t^T \E_\tau \big[ |\Dl_\eps^k Y_s|^{p-4}\sum_{j=1}^n|(\Dl_\eps^k  Z^j_s,\Dl_\eps^k  Y_s)|^2\big]\,ds \\
+ p\int_t^T\E_\tau \big[ |\Dl_\eps^k  Y_s|^{p-2}|\Dl_\eps^k  Z_s|^2\big] ds  =\E_\tau\big[ | \nab^{\eps,k} h_T\Dl_\eps^k  X_T|^p\big] \\
+ 2p\int_t^T\E_\tau \big[|\Dl_\eps^k  Y_s|^{p-2}(\nab^{\eps,k}_2\bar F_s\Dl_\eps^k  X_s + \nab^{\eps,k}_3 \bar F_s\Dl_\eps^k  Y_s,\Dl_\eps^k  Y_s)\big]\,ds.
}
From here, by the forward SDE in \rf{fbsde-diff-1} and Young's inequality,
it follows that a.s.
$\E_\tau |\Dl_\eps^k Y_t|^p \lt \gm_4 \big( 1+ \int_\tau^T \E_\tau |\Dl_\eps^k Y_t|^p \, ds\big)$ for all $t\in [\tau,T]$ and $\eps\gt 0$,
which,
together with the forward SDE in \eqref{fbsde-diff-1}, implies \rf{dif-bounds}.

 Now let $\zeta_X(t)=  \Dl^k_\eps X_t - \Dl^k_{\eps'} X_t $. Similarly, we define $\zeta_Y(t)$  and $\zeta_Z(t)$. The FBSDE for
 the triple $(\zeta_X(t), \zeta_Y(t), \zeta_Z(t))$ takes the form
 \aaa{
 \lb{fbsde-d11}
  \begin{cases}
\zeta_X(t)= \int_\tau^t  \big(\zeta_Y(s)+ \nab^{\eps,k} \bar\eta_s\, \zeta_X(s)+ \xi^X_s \big)  ds,  \\
\zeta_Y(t) = \nab^{\eps,k} h_T\, \zeta_X(T) + \sig_T + \int_t^T  \big(\nab^{\eps,k}_2 \bar F_s\, \zeta_X(s)\\
 \hspace{9mm}  +\nab^{\eps,k}_3 \bar F_s \, \zeta_Y(s)
+ \xi^Y_s \big)ds  
- \int_t^T \zeta_Z(s)dW_s,  
\end{cases}
 }
 where    $\xi^X_s = (\nab^{\eps,k} \bar\eta_s -\nab^{\eps'\hs,k} \bar\eta_s )\Dl^k_{\eps'} X_s$,
 $\xi^Y_s  =   (\nab^{\eps,k}_2 \bar F_s - \nab^{\eps'\hs,k}_2 \bar F_s)\Dl^k_{\eps'} X_s
 + (\nab^{\eps,k}_3 \bar F_s - \nab^{\eps'\hs,k}_3 \bar F_s)\Dl^k_{\eps'} Y_s$,  and
 $\sig_T = (\nab^{\eps,k} h_T - \nab^{\eps'\hs,k} h_T)\Dl^k_{\eps'} X_T$.
 Note that $\nab^{\eps,k} \bar\eta_s$ and $\nab^{\eps,k} h_T$ are bounded by $K$, 
 and $\nab^{\eps,k}_i \bar F_s$, $i=2,3$, are bounded by $K(1+\dl)$,
 which follows from \rf{f7}. Then, by standard arguments
(which include an application of It\^o's formula to $|\zeta_Y|^2$, elevating 
the both parts to the power $\frac{p}2$, and making use of the estimate
$\E \big|\int_t^T (\zeta_Y(s),\zeta_Z(s)dW_s)\big|^{\frac{p}2}
\lt \gm_5(T-\tau)^{\frac{p}4}\E\sup_{[\tau,T]}|\zeta_Y|^p 
+\eps \E\big(\int_t^T |\zeta_Z(s)|^2 ds\big)^\frac{p}2$),
  there exists a constant $\check \gm_k<\dot\gm_K$ such that
   on the interval $[\tau,T]$ 
 whose length is smaller than $\check \gm_{K}$, for $p\gt 2$, 
 \mmm{
 \lb{f8}
\E \sup_{[\tau,T]}|\zeta_X(t)|^p +  \E\sup_{[\tau,T]} |\zeta_Y(t)|^p  
+ \E \big(\int_t^T|\zeta_Z(s)|^2 ds\big)^\frac{p}2  \\
 \lt \gm_6 \big(\E \,|\sig_T|^p + 
 \E \int_\tau^T  [|\xi^X_s|^p +  |\xi^Y_s|^p] ds \big)\lt \gm_7\, |\eps-\eps'|^p. 
}
The last inequality holds by the definition of $\sig_T$, $\xi^X_s$, 
$\xi^Y_s$, and by virtue of \rf{f7-} and \rf{cont7}. Combining \rf{f8} with 
Corollary A.6 from \cite{delarue}, we obtain that there exists a positive constant $\hat \gm_K < \check\gm_K$ such that
for all $x,x'\in\Rnu^n$, $\tau,\tau'\in  [T-\hat\gm_K, T]$, and $t\in [\tau,T]$, 
\mmm{
\lb{cont9}
\E\sup_{t\in[\tau,T]}|\Dl^k_\eps X^{\tau,x}_t - \Dl^k_{\eps'} X^{\tau'\hs,x'}_t|^p + 
\E\sup_{t\in[\tau,T]}|\Dl^k_\eps Y^{\tau,x}_t - \Dl^k_{\eps'} Y^{\tau'\hs,x'}_t|^p \\
\lt \gm_8 (|\eps-\eps'|^p + |x-x'|^p + |\tau - \tau'|^{\frac{p}2}).
}
By Kolmogorov's continuity criterium, there exists a continuous version of the map
$[0,+\infty) \x [T-\hat\gm_K, T] \x [T-\hat\gm_K, T]\x \Rnu^n \to \Rnu^{2n}$, $(\eps, \tau, t, x) 
\mto(\Dl^k_\eps X^{\tau,x}_t, \Dl^k_\eps Y^{\tau,x}_t)$. 
This means
that the map $[T-\hat\gm_K, T]\x [T-\hat\gm_K, T] \x \Rnu^n \to \Rnu^{2n}$, $(\tau,t,x) \mto (X^{\tau,t,x}_t, Y^{\tau,t,x}_t)$ is differentiable in $x_k$, 
and the derivative is continuous in $(\tau,t,x)$ a.s. 
In particular, there exists an a.s. continuous derivative $\pl_k Y^{\tau,x}_\tau$,
and, by \rf{dif-bounds}, a.s.,
 \aaa{
 \lb{yd11}
 |\pl_{k} Y^{\tau,x}_\tau| <  \big(\gm_3\big)^{\frac1{p}} \quad \text{for all} \;
 (\tau,x)\in  [T-\hat\gm_{K},T] \x \Rnu^n,
 }
 where $\pl_k = \pl_{x_k}$.
 This holds for all $k\in \{1,\ldots, n\}$.
 Moreover, $\gm_3$ does not depend on $\nu$.

{\it Step 4.  Second order differentiability of the FBSDE solution in $x$. Boundedness of $\E_\tau|\pl^2_{ik} Y^{\tau,x}_t|^2$}.
Below, we use the symbol $\pl_k$ for $\pl_{x_k}$ and $\pl^2_{ik}$ for $\pl^2_{x_i x_k}$.
As in Step 3, we write $F$ instead of $F_\dl$ to simplify notation.

Remark that  \rf{f8} implies the differentiability in $x$ of   
$Z^{\tau,x}_t$  with respect to the norm $(\E\int_\tau^T |\ffi(s)|^2 ds)^\frac12$.
Further, the FBSDE for $(\pl_k X_t, \pl_k Y_t, \pl_k Z_t)$ takes the form:
 \aaa{
\lb{fbsde-d-short}
\phantom{kkk
kkk
}
 }
   \vspace{-12mm} 
\aa{
\hspace{9mm}
\begin{cases}
  \pl_k X_t =  e_k + \int_\tau^t \big(\pl_k Y_s + \nab \bar\eta_s \pl_k X_s\big)  ds \\
 \pl_k Y_t = \nab h (X_T)\pl_k X_T +
\int_t^T\big( \nab_2 \bar F_s\pl_k X_s + \nab_3 \bar F_s\pl_k Y_s \big) ds - \int_t^T \pl_k Z_s \,dW_s,
\end{cases}
 }
where $\nab \bar\eta_s = \nab \bar\eta(s,X_s)$, $\nab h_T = \nab h(X_T)$, $\nab_i \bar F_s = \nab_i \bar F (s,X_s,Y_s)$,  $i=2,3$.

As in the previous step, define $\Dl^i_\eps \pl_k X_t=\eps^{-1} 
(\pl_k X_t^{\tau,x+\eps e_i} - \pl_k X_s^{\tau,x})$, $i=1, \ldots, n$, and,
similarly, $\Dl^i_\eps \pl_kY_t$, $\Dl^i_\eps \pl_k Z_t$. Applying the operation 
$\Dl^i_\eps$ to FBSDE \rf{fbsde-d-short}, using  formula
\rf{tayl}, and noticing that for any  functions $\al_1(x)$ and $\al_2(x)$,
$\Dl^i_\eps\big[\al_1(x)\al_2(x)\big] =  \al_1(x)\Dl^i_\eps \al_2(x) + \Dl^i_\eps \al_1(x) \al_2(x+\eps e_i)$, we obtain
the FBSDE
for the triple $(\Dl^i_\eps \pl_k X_t, \Dl^i_\eps \pl_k Y_t,  \Dl^i_\eps \pl_k Z_t)$
\aaa{
  \label{fbsde-2-der}
\begin{cases}
  \Dl^i_\eps \pl_k X_t = - \int_\tau^t \big( \Dl^i_\eps \pl_kY_s+ \nab \bar\eta_s \,  \Dl^i_\eps \pl_k X_s +   \tet^X_{s,\eps} \big) ds,
 \\
 \Dl^i_\eps \pl_kY_t  = \nab h_{T}  \Dl^i_\eps \pl_k X_T + \eta_{T,\eps}
+  \int_t^T  \hs \big(\nab_2 \bar F_s  \Dl^i_\eps \pl_k X_s \\ \hspace{12mm}  + \nab_3 \bar F_s  \Dl^i_\eps \pl_kY_s   + \tet^Y_{s,\eps}\big) ds  
- \int_t^T \hs \Dl^i_\eps \pl_kZ_s  \,dW_s ,
\end{cases}
 }
 where 
  \lb{chi22}
 \een{
  \tet^X_{s,\eps} =  \nab^{\eps,i} \nab \bar\eta_s \, \Dl^i_\eps X_s \,
  \pl_k X_s^{\tau, x+\eps e_i}; \quad
   \eta_{T,\eps} = \nab^{\eps,i} \nab h_T \, \Dl^i_\eps X_T \,
    \pl_k X^{\tau, x+\eps e_i}_T;  \\
   \tet^Y_{s,\eps} =    \nab^{\eps,i}_2\nab_2 \bar F_s \Dl^i_\eps X_s \pl_k X^{\tau, x+\eps e_i}_s 
+ \nab^{\eps,i}_3\nab_3 \bar F_s \Dl^i_\eps Y_s \pl_k Y^{\tau, x+\eps e_i}_s \\ 
\hspace{5mm}+  \nab^{\eps,i}_3\nab_2 \bar F_s \Dl^i_\eps Y_s \pl_k X^{\tau, 
x+\eps e_i}_s  
   +\nab^{\eps,i}_2\nab_3 \bar F_s \Dl^i_\eps X_s \pl_k Y^{\tau, x+\eps e_i}_s.
 }
 Further, the triple $(\Dl^i_0 \pl_k X_t, \Dl^i_0 \pl_k Y_t,  \Dl^i_0 \pl_k Z_t)$
 will denote the unique solution to FBSDE \rf{fbsde-2-der} whose
 coefficients $\tet^X_{s,\eps}$, $\eta_{T,\eps}$, and $\tet^Y_{s,\eps}$ 
 are taken at $\eps = 0$. The existence and uniqueness of the above triple
 follows from Theorem A.1 in \cite{delarue}.
Let us show that, a.s.,  
 \aaa{
\lb{21-bound}
\max\{\E_\tau |\Dl^i_\eps \pl_k X_t|^2, \E_\tau |\Dl^i_\eps \pl_k Y_t|^2\} \lt \mu_1
\quad \text{for all} \; \eps\gt 0, \, t\in [\tau,T].
}
It\^o's formula implies
 \mm{
| \Dl^i_\eps \pl_kY_t |^2
+ \int_t^T  | \Dl^i_\eps \pl_kZ_s|^2 ds  =
 |\nab h(X_T) \Dl^i_\eps \pl_k X_T +  \eta_{T,\eps}|^2 
+ 2\int_t^T ( \nab_2 \bar F_s \,  \Dl^i_\eps \pl_kX_s  \\ 
+  \nab_3 \bar F_s\,  \Dl^i_\eps \pl_kY_s  
  + \tet^Y_{s,\eps}, \Dl^i_\eps \pl_kY_s)\,ds
+ \int_t^T  ( \Dl^i_\eps \pl_kY_s, \Dl^i_\eps \pl_kZ_s dW_s).
}
From here, by using the forward SDE in \rf{fbsde-2-der}, we conclude that
there exists a constant $\bar \gm_{K}< \hat \gm_{K}$, depending only on $K$, such that for $\tau\in [T-\bar \gm_{K} ,T]$,
\aa{
\E_\tau |\Dl^i_\eps \pl_kY_t|^2 \lt \mu_2 \big( 1+ 
 \int_\tau^T  \E_\tau \big(|\tet^X_{s,\eps}|^2 +  |\tet^Y_{s,\eps}|^2\big) ds +
 \E_\tau |\eta_{T,\eps}|^2 \big) \quad \text{a.s.}
}
By the assumptions of the theorem and \rf{dif-bounds}, the right-hand side of the above inequality is bounded a.s.
This implies 
\rf{21-bound}.

Now let us prove the existence of a continuous second derivative of the map
$Y^{\tau,x}_\tau$.
Let $\zeta_X(t)=  \Dl^i_\eps \pl_k X_t -  \Dl^i_{\eps'} \pl_k X_t$, $\zeta_Y(t)=\Dl^i_\eps \pl_k Y_t -  \Dl^i_{\eps'} \pl_k Y_t$, 
$\zeta_Z(\eps,t)=\Dl^i_\eps \pl_k Z_t -  \Dl^i_{\eps'} \pl_k Z_t$. 
The FBSDE for the triple $(\zeta_X(t), \zeta_Y(t), \zeta_Z(t))$ takes the form:
\aaa{
  \label{fbsde-2-zeta}
\begin{cases}
\zeta_X(t) =  -\int_\tau^t \big(\zeta_Y(s) + \nab\bar\eta_{s} \zeta_X(s) + \tet^X_{s,\eps}-\tet^X_{s,\eps'}) \, ds, \\
\zeta_Y(t)   = \nab h_{T} \zeta_X(T) + \eta_{T,\eps} - \eta_{T,\eps'}
+ \int_t^T \big( \nab_2 \bar F_{s} \zeta_X(s) \\ \hspace{9mm} + \nab_3 \bar F_{s} \zeta_Y(s) + \tet^Y_{s,\eps} -  \tet^Y_{s,\eps'}\big) ds  
-  \int_t^T \zeta_Z(s)  \,dW_s. 
\end{cases}
 }
Note that FBSDE \rf{fbsde-2-zeta} has a similar structure 
with FBSDE \rf{fbsde-d11}. 
Thus, similar to \rf{f8}, we conclude that there exists a constant $\mathring \gm_K<\bar \gm_K$ such that
for $\tau\in [T-\mathring \gm_{K} ,T]$,
 \mmm{
 \lb{f9}
\E |\zeta_X(t)|^p +  \E |\zeta_Y(t)|^p  + 
\E \Big(\int_t^T|\zeta_Z(s)|^2 ds \Big)^\frac{p}2 \\
 \lt \mu_3 \big(\E \,|\eta_{T,\eps} - \eta_{T,\eps'}|^p + 
 \E \int_\tau^T  [|\tet^X_{s,\eps}-\tet^X_{s,\eps'}|^p +  | \tet^Y_{s,\eps} -  \tet^Y_{s,\eps'}|^p] ds \big)\lt \mu_4\, |\eps-\eps'|^p
}
on  $[\tau,T]$.
The last inequality holds by \rf{cont7} and \rf{cont9}.
 Combining \rf{f9} with 
Corollary A.6 from \cite{delarue} (similar to the previous step), we obtain that there exists a positive constant 
$\gm_K < \mathring \gm_K$ such that
for all $x,x'\in\Rnu^n$, $\tau,\tau'\in  [T-\gm_K, T]$, and $t\in [\tau,T]$, 
\aa{
\E|\Dl^i_\eps \pl_k Y^{\tau,x}_t - \Dl^i_{\eps'} \pl_k Y^{\tau'\hs,x'}_t|^p 
\lt \mu_5 (|\eps-\eps'|^p + |x-x'|^p + |\tau - \tau'|^{\frac{p}2}).
}
By Kolmogorov's continuity criterium, there exists a continuous version of the map
$[0,+\infty) \x [T-\gm_K, T] \x \Rnu^n \to \Rnu^n$, $(\eps, \tau, x) \mto \Dl^i_\eps \pl_kY^{\tau,x}_\tau$. This means
that the map $[T-\gm_K, T] \x \Rnu^n \to \Rnu^n$, $(\tau,x) \mto \pl_k Y^{\tau,x}_\tau$ is differentiable in $x_i$ 
and the derivative in continuous in $(\tau,x)$ a.s.
Further, \rf{21-bound} implies that 
\aaa{
\lb{2-bound}
|\pl^2_{ik} Y^{\tau,x}_\tau| \lt \sqrt{\mu_2} \quad \text{a.s.}
}
We remark that $\mu_2$ depends only on $K$ and does not depend on $\nu$. Moreover, \rf{2-bound}
holds uniformly in $(\tau,x) \in  [T- \gm_{K},T]\x \Rnu^n$ by continuity.
 This implies that there exists a set $\td\Om$ of full $\PP$-measure 
 such that for all $\om\in\td \Om$, $Y^{\tau,x}_\tau$ is twice continuously differentiable in $x$,
 and, moreover, the derivatives of  $Y^{\tau,x}_\tau$  up to the second order are bounded.

\textit{ Step 5. Solution to random PDE \rf{backward-pde}.}
Define $\bar y(\tau,x,\om) =  Y^{\tau,x}_\tau(\om)$ for each $\om\in \td\Om$. 
Note that $\bar y(\tau,x)$  is $\mc F_{T-\tau}$-measurable and by \rf{rel1}, a.s.,
\aaa{
\lb{144}
Y^{\tau,x}_t = \bar y(t, X^{\tau,x}_t) \quad \text{for all} \; 
\tau,t \in [T-\gm_K,T], \, x\in\Rnu^n.
}
Let us prove that $\bar y(t,x)$ is a solution to \rf{backward-pde}. The idea of the proof is similar to that of Theorem 3.2 in \cite{peng92}.
However, we deal with the random coefficient case. 
Define $\mc L u = \nu\lap u + (u +\bar \eta,\nab)u$.
We have
\aa{
\bar y(t+h,x) - \bar y(t,x) = [\bar y(t+h,x) - \bar y(t+h,X^{t,x}_{t+h})] + [\bar y(t+h,X^{t,x}_{t+h}) -\bar y(t,x)].
}
Since $\bar y$ is of class $\C^{0,2}_b$, we can apply
It\^o's formula to the first term.
Further, by \rf{fbsde-modified} and  \rf{144}, we substitute the second term  with 
$-\int_t^{t+h} \bar F_\dl(s,X^{t,x}_s, \bar y(s, X^{t,x}_s)ds + \int_t^{t+h} Z^{t,x}_s dW_s$. Remark that, by \rf{y-bound},
$\bar F_\dl(s,X^{t,x}_s, \bar y(s, X^{t,x}_s)) = \bar F(s,X^{t,x}_s, \bar y(s, X^{t,x}_s))$ so we can skip the index $\dl$. Thus, we obtain that, a.s.,
\mm{
\bar y(t+h,x) - \bar y(t,x) = -\int_t^{t+h} \mc L \bar y(t+h,X^{t,x}_s) ds
- \sqrt{2\nu}\int_t^{t+h} \nab \bar y(t+h,X^{t,x}_s)dW_s \\
- \int_t^{t+h} \bar Fl(s,X^{t,x}_s, \bar y(s, X^{t,x}_s))ds 
+ \int_t^{t+h} Z^{t,x}_s dW_s
}
for all $(t,x,h)$.
Fix a partition $\mc P= \{\tau = t_0 < t_1 < \dots < t_n = T\}$. Taking the conditional expectation $\E_\tau$ 
and summing up, we obtain that, a.s.,
\aaa{
\lb{p4}
\bar y(\tau,x)- h(x) = \E_\tau \sum_{i=0}^{n-1} \int_{t_i}^{t_{i+1}} \big(\mc L \bar y(t_{i+1},X^{t_i,x}_s)
+\bar F(s,X^{t_i,x}_s, \bar y(s, X^{t_i,x}_s))\big)ds. 
}
Indeed, the conditional expectation of the stochastic integrals is zero by Lemma \ref{lem2}.
Note that the expression under the integral sign is bounded, a.s.,
since $\mc L\bar y(t,x)$ is bounded by what was proved in the previous steps. 

Further, $\mc L \bar y(t,x)$ and $\bar F (s,X^{t,x}_s, \bar y(s, X^{t,x}_s))$ are a.s. continuous in $(t,x)$.
Letting the mesh of $\mc P$ in \rf{p4} go to zero, by the conditional bounded convergence theorem, we obtain
that $\bar y(t,x)$
solves \rf{backward-pde} on $[T-\gm_K,T]\x \Rnu^n$.
Further, by \rf{y-bound}, \rf{yd11}, \rf{2-bound}, and by equation \rf{backward-pde} itself, we conclude that, a.s., $\bar y \in\C^{1,2}_b$.
Finally, as we have already mentioned in Step 1, $\bar y$ is $\mc F_{T-\tau}$-adapted for each $x\in\Rnu^n$.
The theorem is proved. 
\end{proof}

 \subsection{Gradient estimate}

 In this section, we present an FBSDE stochastic method to obtain a uniform in $r$ bound for the gradient $\pl_x y(t,x)$
of the solution $y(t,x)$ to the following final value problem:
\aaa{
\lb{pde1}
\begin{cases}
\pl_t y(t,x) + \frac12\tr(\pl^2_{xx} y(t,x)\sg(t,x)\sg(t,x)^\top) \\ + (\ffi(t,x,y(t,x)),\pl_x)y(t,x) 
 +f(t,x,y(t,x), \pl_x y(t,x)\sg(t,x,y)) = 0,\\
y(T,x) = h(x), \quad x\in \Rnu^n, \; t \in [r,T], \; r\gt 0.
\end{cases}
}
Here $\sg(t,x)^\top$ is the transpose to the matrix $\sg$,
  $\tr(\pl^2_{xx} y(t,x)\sg(t,x)\sg(t,x)^\top)$ is the vector whose $l$-th component is
 the trace of the matrix  $\pl^2_{xx} y_l(t,x)\sg(t,x)\sg(t,x)^\top$, where $y_l(t,x)$ is the $l$-th component
 of $y(t,x)$, and $(\ffi(t,x,y(t,x)),\pl_x)$ is the formal scalar product of $\ffi$ and the vector $\pl_x$ with the coordinates
 $(\frac{\pl}{\pl x_1}, \frac{\pl}{\pl x_2}, \dots, \frac{\pl}{\pl x_n})$.
Equation \rf{pde1} is assumed to be $\Rnu^m$-valued, $\sg(t,x)$,
$\ffi(t,x,y)$, and $f(t,x,y,z)$ 
take values in $\Rnu^{n\x n}$, $\Rnu^n$, and $\Rnu^m$, respectively,
and the arguments of these functions are of appropriate dimensions.

It is well known that the FBSDE associated to \rf{pde1} takes the form (see e.g. \cite{delarue})
\aaa{
\lb{fbsde-grl}
\begin{cases}
X^{\tau,x}_t = x + \int_\tau^t \ffi(s,X^{\tau,x}_s,Y^{\tau,x}_s) ds + \int_\tau^t \sg(s,X^{\tau,x}_s)dW_s,\\
Y^{\tau,x}_t = h(X^{\tau,x}_T) + \int_t^T f(s,X^{\tau,x}_s,Y^{\tau,x}_s, Z^{\tau,x}_s)ds -\int_t^T Z^{\tau,x}_s dW_s,
\end{cases}
}
where $\tau \in [r,T]$, $W_t$ is an $n$-dimensional Brownian motion.


Consider a probability space $(\Om,\mc F, \PP)$, and
for each fixed $\tau\in [0,T]$, define the filtration $\mc F^W_{\tau,t}=\sg\{W_s-W_\tau, s \in [\tau,t]\}  \vee \mc N$, where $\mc N$ is the collection
of $\PP$-null sets.  The solution $(X^{\tau,x}_t,Y^{\tau,x}_t,Z^{\tau,x}_t)$ to \rf{fbsde-grl} is understood
in the same way as  in \cite{delarue}.

In the remainder of this section, we make use of the following assumptions.
\bi
%

\item[\bf (B1)] The functions $f$, $\ffi$, $\sg$, and $h$,
 are differentiable with respect to their spatial variables;
the derivatives $\pl_x \sg$ and $\nab h$ are bounded by a constant $K$,
 and the  other derivatives satisfy the linear growth condition
 on $[0,T]\x \Rnu^n\x \Rnu^m \x \Rnu^{m\x n}$:
 \aa{
 |\pl_{(x,y)}\ffi| +  |\pl_{(x,y,z)}f| \lt K(1+|y|). 
 }

\item[\bf (B2)] Assume  there exists a constant $L>0$
such that for all $(t,x,y,z)\in [0,T]\x \Rnu^n\x \Rnu^m \x \Rnu^{m\x n}$, 
\aa{
&|h(x)| + |\sg(t,x)| \lt L;  \quad  |\ffi(t,x,y)| \lt L(1+|x| +|y|);\\
&|f(t,x,y,z)| \lt L(1 + |y| + |z|).
}
\item[\bf (B3)] Finally, assume there exists a constant $\la>0$ such that for all $(t,x)\in [0,T] \x \Rnu^n$ and $\zeta\in\Rnu^n$,
\aa{
(\sg(t,x)\sg(t,x)^\top\zeta,\zeta)>\la |\zeta|^2.
}
\ei

\begin{lem}
\lb{lem3}
Assume $y(t,x)$ is a $\C^{1,2}_b([0,T]\x\Rnu^n)$-solution to final value problem \rf{pde1} on $[r,T]\x \Rnu^n$.
Then, for any $\tau\in [r,T]$,
\aaa{
\lb{form}
\big(X^{\tau,x}_t, y(t,X^{\tau,x}_t), \pl_x y(t,X^{\tau,x}_t)
\sg(t, X^{\tau,x}_t)\big)
}
is a solution to FBSDE \rf{fbsde-grl} on $[\tau,T]$.
\end{lem}
\begin{proof}
The existence and uniqueness of solution to the SDE
\aaa{
\lb{fsde}
X^{\tau,x}_t = x + \int_\tau^t \ffi(s,X^{\tau,x}_s, y(s, X^{\tau,x}_s) )ds + \int_\tau^t \sg(s, X^{\tau,x}_s)\, dW_s
}
is a classical result under (B1) and (B2).

Now assume that $(X^{\tau,x}_t,Y^{\tau,x}_t,Z^{\tau,x}_t)$ is given by \rf{form}. Then, the forward SDE in \rf{fbsde-grl} is satisfied.
Applying It\^o's formula to $y(t,X^{\tau,x}_t)$ at times $t$ and $T$, we can easily check that the above triple
verifies the backward SDE in \rf{fbsde-grl}.
\end{proof}

 Our main result in this subsection is the following.
\begin{thm}
\lb{grad-est}
Assume (B1)--(B3).
Further assume that $y(t,x)$ is a $\C^{1,2}_b$-solution to final value problem \rf{pde1} on $[r,T]\x \Rnu^n$.
Then,  there exists a constant $\gm_{T,K, L, \la}$,
that depends only on $T$, $K$, $L$, and $\la$,
such that for all  $(x,t)\in \Rnu^n \x [r,T]$,
\aaa{
\lb{bound}
|\pl_x y(t,x)| \lt \gm_{T,K, L, \la}.
}
In particular, the constant $\gm_{T,K, L, \la}$ does not depend on $r$.
%
\end{thm}
\begin{proof}
Everywhere throughout the proof, $\gm^{(i)}_{\mc A}$, $i=1,2,\ldots$,  will denote constants depending only on the set 
of parameters $\mc A$.

\textit{Step 1. Boundedness of $y(t,x)$.}
Let $(X_t^{\tau,x},  Y_t^{\tau,x},  Z_t^{\tau,x})$
be the solution to \rf{fbsde-grl} on $[\tau,T]$ given by \rf{form}.
For simplicity of notations, in what follows, we skip the upper index ${\tau,x}$
using it just where it is necessary. 

It\^o's formula and the backward SDE in \rf{fbsde-grl} imply  
\aaa{
\lb{gen-case}
\E |Y_t|^2 + \int_t^T \E|Z_s|^2 ds = \E|h(X_T)|^2 + \E\int_t^T 2(f(s,X_s,Y_s,Z_s),Y_s)ds.
}
By Assumption (B2), there exists a constant $\gm^{(1)}_L$ such that
\aa{
\E |Y_t|^2 + \int_t^T \E|Z_s|^2 ds \lt L^2 + \gm^{(1)}_{L}\int_t^T\E |Y_s|^2 ds
+ \frac12 \int_t^T \E |Z_s|^2 ds.
}
By Gronwall's inequality,  for all $t\in [\tau,T]$, 
\aa{
\E |Y_t|^2 \lt \gm^{(2)}_{L,T}. 
}
Since $Y^{\tau,x}_t = y(t,X^{\tau,x}_t)$, where $X^{\tau,x}_t$ is the unique solution to \rf{fsde}, then
\aaa{
\lb{bound-y}
|y(\tau,x)| \lt M_{L,T},
}
where $M_{L,T}$ is a constant that depends only on $L$ and $T$. 
 
{\it Step 2. Transformation of the PDE.}
Rewrite PDE \rf{pde1} with respect to
\aaa{
\lb{new-f}
\td y(t,x) = \frac1{\al}\, y (t,x),
}
where $\al= 3M_{L,T}$.
We obtain
\aaa{
\lb{pde2}
\begin{cases}
\pl_t \td y(t, x) + \frac12\tr(\pl^2_{xx} \td y(t, x)(\sg\sg^\top)(t,x)) 
+\big(\ffi(t,x,\al\,\td y(t,x)),\pl_x\big)\td y(t,x) \\
+\frac1{\al}f\big(t,x,\al y(t,x), \al\, \pl_x \td y(t,x)\sg(t,x)\big) = 0,\\
\td y(T,x) = \frac1{\al}h(x).
\end{cases}
}
Let $X_t$ be the solution to SDE \rf{x-sol1} below 
\aaa{
\lb{x-sol1}
X_t = x + \int_\tau^t \ffi(s, X_s,\al \td y(s, X_s)) ds + \int_\tau^t \sg(s,X_s)dW_s.
}
By Lemma \ref{lem3}, the triple
\aaa{
\lb{form1}
X_t, \quad Y_t = \td y(t,X_t), \quad Z_t = \pl_x \td y(t,X_t) \sg(t, X_t)
}
is the solution to the associated FBSDE
\aaa{
\lb{fbsde-grl2}
\begin{cases}
X_t = x + \int_\tau^t  \ffi(s,X_s,\al Y_s) ds  + \int_\tau^t  \sg(s,X_s)dW_s,\\
Y_t = \frac1{\al} h( X_T) + \int_t^T \frac1{\al} f(s, X_s,\al Y_s, \al Z_s)ds -\int_t^T Z_s dW_s.
\end{cases}
}
Although the solution triple, defined by \rf{form1}, is different than the triple defined by \rf{form}--\rf{fsde},
we denote it again by $(X_t,Y_t,Z_t)$ for simplicity of notation.

{\it Step 3. Boundedness of $\E\exp{\Big\{\frac{\la}4 \int_\tau^T |\nab\td y(s,X_s)|^2 ds\Big\}}$.}
Note that \rf{bound-y} and \rf{new-f} imply that $|\td y(\tau,x)| \lt  \frac13$ for all $\tau\in [r,T]$
by the choice of $\al$, and, therefore,
by \rf{form1},
\aaa{
\lb{bound-td-y}
 |Y_t| \lt \frac13 \quad \text{for all} \; t\in [\tau,T] \quad \text{a.s.}
}
By It\^o's product formula and \rf{fbsde-grl2}, we obtain
\mm{
|Y_t|^2 + \int_t^T |Z_s|^2 ds = \frac1{\al^2}|h(X_T)|^2 +
2\int_t^T(\frac1{\al}f(s, X_s,\al Y_s, \al Z_s), Y_s)ds\\
+2\int_t^T(Y_s,Z_sdW_s)
\lt \gm^{(3)}_{L,\al}\Big(1 + \int_t^T|Y|_s\, ds +  \int_t^T |Y|^2_s \,ds\Big)
+\frac12\int_t^T|Z_s|^2 ds\\
+2\int_t^T(Y_s,Z_sdW_s).
}
By \rf{bound-td-y}, there exists a constant $\gm^{(4)}_{L,T}$ such that
\aa{
\frac12 \int_t^T |Z_s|^2 ds \lt \gm^{(4)}_{L,T} +2\int_t^T(Y_s,Z_sdW_s).
}
This implies
\mm{
\exp{\Big\{\frac12 \int_t^T |Z_s|^2 ds\Big\}} \lt \gm^{(5)}_{L,T} \,
\exp{\Big\{2\int_t^T(Y_s,Z_sdW_s) - 2\sum_{i=1}^n \int_t^T (Y_s,Z^i_s)^2 ds\Big\}}\\
\x\exp{\Big\{\frac{2}{9} \int_t^T |Z_s|^2 ds\Big\}}.
}
Therefore, 
\aaa{
\lb{DD}
\exp{\Big\{\frac14 \int_t^T |Z_s|^2 ds\Big\}} \lt \gm^{(5)}_{L,T} \, \exp{\Big\{2\int_t^T(Y_s,Z_sdW_s) -
2\sum_{i=1}^n \int_t^T (Y_s,Z^i_s)^2 ds\Big\}}.
}
Note that on the right-hand side we have a Dol\'eans-Dade exponential of a martingale
considered as a process with respect to $T$ while $t$ is fixed.
Indeed, by (B2) and \rf{form1}, the Novikov condition
$\E\big[\exp\{ \sum_{i=1}^n \int_t^T (Y_s,Z^i_s)^2 ds\}\big]<\infty$ is fulfilled.
Therefore, the  expectation  of the exponential on
the right-hand side of \rf{DD} equals to one.
Finally, representation \rf{form1} for $Z_s$ and (B3) imply 
\aaa{
\lb{compare}
\E\exp{\Big\{\frac{\la}4 \int_\tau^T |\nab\td y(s,X_s)|^2 ds\Big\}}
\lt \gm^{(5)}_{L,T}.
}

{\it Step 4. Obtaining an a priori bound for $\pl_x y(t,x)$.}
Since any solution to the final value problem \rf{pde1} is bounded by $M_{T,L}$, introduce $\hat \ffi$ and $\hat f$ as follows
\aa{
\hat \ffi(t,x,y) = \ffi(t,x,y\xi_{M_{T,L}}(y)) \quad \text{and} \quad \hat f(t,x,y,z) = f(t,x,y\xi_{M_{T,L}}(y),z),
}
where $\xi_{M_{T,L}}(y)$ is a $\C^\infty$-cutting function for the ball $B_{M_{T,L}}$ introduced in Remark \ref{rem2.2}.
Note that by (B1), $\hat\ffi$ and $\hat f$ 
possess bounded derivatives w.r.t. the spacial variables. 
Let $\gm^{(6)}_{K,L,T}$ be the common bound for these spatial derivatives.
This bound depends on $K$, and  on $T$, $L$ via  the constant $M_{T,L}$. Observe that the solution
$(X_t,Y_t,Z_t)$ to FBSDE \rf{fbsde-grl2},  given by \rf{form1}, is also  a solution to
\aaa{
\lb{fbsde-grl3}
\begin{cases}
X_t = x + \int_\tau^t  \hat\ffi(s,X_s,\al Y_s) ds  + \int_\tau^t  \sg(s,X_s)dW_s,\\
Y_t = \frac1{\al} h( X_T) + \int_t^T \frac1{\al} \hat f(s, X_s,\al Y_s, \al Z_s)ds -\int_t^T Z_s dW_s.
\end{cases}
}
Let $(\pl_x X_s, \pl_x Y_s, \pl_x Z_s)$ denote the derivative of the solution to
 FBSDE \rf{fbsde-grl3} w.r.t. the initial data $x$. Further, for the function $\hat f(t,x,y,z)$, $\nab_2 \hat f = \pl_x \hat f$, 
 $\nab_3 \hat f = \pl_y \hat f$, and $\nab_4 \hat f = \pl_z \hat f$. For the function $\hat\ffi$, the derivatives $\nab_2$ and $\nab_3$
 are defined similarly. In case of just one spatial variable, as in the function $\sg$, we skip the index $2$.
Remark that under (B1)-(B2), the differentiability of the solution $X^{\tau,x}_t$ to SDE \rf{fsde} is well known 
and the derivative process satisfies
\aa{
\pl_x X_{t} = I + \int_{\tau}^{t} \nab\td \ffi_s \pl_x X_s ds +
\int_{\tau}^{t} \nab\sg_s\pl_x X_s dW_s,
}
where 
 \aa{
 \td \ffi(t,x) = \hat\ffi(t, x,\al \td y(t,x)),
 }
$\td\ffi_s$ and $\sg_s$ are abbreviations for $\td\ffi(s,X_s)$  and $\sg(s,X_s)$, respectively.
An application of It\^o's formula gives
 \mmm{
 \lb{m6}
|\pl_x X_t|^2 = 1 + 2\int_\tau^t (\nab \td \ffi_s\pl_x X_s, \pl_x X_s) ds 
 + 2\sum_{k=1}^n\int_\tau^t (\nab   \sg^k_s\pl_x X_s,\pl_x X_s) dW^k_s \\
 + \sum_{k=1}^n\int_\tau^t  |\nab  \sg^k_s \pl_x X_s|^2 ds,
 } 
 where $\sg^k_s = (\sg_s,e_k)$.
 Define
 \aa{
 \tet_s =
 \begin{cases}
 \frac{\pl_x X_s}{|\pl_x X_s|} \quad \text{if} \; \pl_x X_s\ne 0,\\
 0 \quad \text{otherwise}.
 \end{cases}
 }
Equation \rf{m6} becomes
 \mm{
  |\pl_x X_t|^2 = 1 + 2\int_\tau^t (\nab \td \ffi_s\tet_s, \tet_s)|\pl_x X_s|^2 ds
 + \sum_{k=1}^n\int_\tau^t  |\nab  \sg^k_s \tet_s|^2  |\pl_x X_s|^2  ds\\
  + 2\sum_{k=1}^n\int_\tau^t (\nab  \sg^k_s\tet_s, \tet_s)|\pl_x X_s|^2 dW^k_s.
 }
This implies the following representation for $|\pl_x X_t|^2$ via the Dol\'eans-Dade exponential:
\aa{
|\pl_x X_t|^2 = &e^{-1}\exp \Big\{
\int_\tau^t \big[ 2(\nab \td \ffi_s\tet_s, \tet_s)+  \sum_{k=1}^n \big( |\nab  \sg^k_s \tet_s|^2
+2(\nab  \sg^k_s \tet_s,\tet_s)^2\big) \big]   ds \Big\}\\
\x & \exp \Big\{ 2\sum_{k=1}^n \int_\tau^t (\nab  \sg^k_s \tet_s,\tet_s)  dW^k_s
- 4\sum_{k=1}^n \int_\tau^t (\nab  \sg^k_s \tet_s,\tet_s)^2 ds
\Big\}.
}
In the above expression,  the term
$2\sum_{k=1}^n\int_\tau^t (\nab \sg^k_s \tet_s,\tet_s)^2 ds$
was added and subtracted so we could get the estimate
\aaa{
\lb{nab-x-est}
|\pl_x X_t|^2 \lt &\exp \Big\{
2 \int_\tau^t \big( 2|\nab \td \ffi_s|+ 3|\nab \sg_s|^2\big) ds \Big\}\\
+ &\exp \Big\{ 4\sum_{k=1}^n \int_\tau^t (\nab  \sg^k_s \tet_s,\tet_s)  dW^k_s
- 8\sum_{k=1}^n \int_\tau^t (\nab  \sg^k_s \tet_s,\tet_s)^2 ds \Big\}. \notag
}
Since  $\nab \td \ffi(t,x) = \nab_2\hat \ffi(t, x,\al \td y(t,x)) +\al \nab_3\hat \ffi(t, x,\al \td y(t,x))\pl_x \td y(t,x)$, 
\aa{
|4\nab \td \ffi_s| \lt 4\gm^{(6)}_{K,L,T}(1+ \al |\nab \td y(s,X_s)|)
\lt  4\gm^{(6)}_{K,L,T} + \frac{16  (\gm^{(6)}_{K,L,T})^2 \al^2}{\la} + \frac{\la}4 \, |\nab \td y(s,X_s)|^2.
}
Taking the  expectation of the both parts of \rf{nab-x-est}, we obtain
\aaa{
\lb{est-nab-x2}
\E|\pl_x X_t|^2 \lt \gm^{(7)}_{K,L,T,\la} \,
\E\exp\Big\{ \frac{\la}{4}  \int_\tau^T \big|\nab \td y(s,X_s)\big|^2 ds   \Big\} + 1 \lt \gm^{(8)}_{K,L,T,\la},
}
where the last inequality holds by \rf{compare}.

Further, let us estimate $\E|\pl_x Y_t|^2$. Applying
It\^o's product formula and using the backward SDE in \rf{fbsde-grl3},
we obtain that
\mmm{
\lb{est2}
\E|\pl_x Y_t|^2 + \int_t^T \E|\pl_x Z_s|^2 ds
 = \frac1{\al^2}\E|\nab h_T\pl_x X_T|^2  \\+
2\int_t^T \E\big(\frac1{\al}\nab_2 \hat f_s \pl_x X_s +
\nab_3 \hat f_s \pl_x Y_s + \nab_4 \hat f_s \pl_x Z_s,\pl_x Y_s\big)
ds
\lt \gm^{(9)}_{K,T, L}\Big(\E|\pl_x X_T|^2 \\+  \int_t^T \E|\pl_x X_s|^2 ds
+ \int_t^T \E|\pl_x Y_s|^2ds\Big) + \frac12\int_t^T \E|\pl_x Z_s|^2 ds.
}
By \rf{est-nab-x2} and Gronwall's inequality,
\aa{
\E|\pl_x Y_t|^2 \lt \gm^{(10)}_{K, L, T, \la}.
}
Evaluating at $t=\tau$,
and taking into account that $\td y$ and $y$ are related by the formula $y(t,x)=\al\td y(t, x)$,
we obtain the final estimate, i.e., there exists a constant $\gm_{K, L, T, \la}$ such that
\aa{
|\pl_x y(\tau,x)| \lt \gm_{K, L, T, \la}.
}
The theorem is proved.
\end{proof}
\subsection{Global existence}
We start with a lemma on the uniqueness of a $\C^{1,2}_b$-solution to Cauchy problem \rf{rforced}.
\begin{lem}
\lb{lem-unique}
Assume (A1)--(A3). Then, problem \rf{rforced} can have at most one 
pathwise $\C^{1,2}_b([0,T]\x\Rnu^n)$-solution on $[0,T]$.
\end{lem}
\begin{proof}
 Assume there are two solutions  $y_1, y_2\in\C^{1,2}_b([0,T]\x\Rnu^n)$
 to problem \rf{rforced},
and let $y=y_1 - y_2$. Then, $y(t,x)$ solves the problem
\aaa{
\lb{linear1}
\begin{cases}
\pl_t y(t,x)   =  \nu \Delta y(t,x) - (\eta(t,x) + y_1,\nabla) y(t,x)  \\  \hspace{2mm} +\big(\Phi(t,x)+\pl_x y_2\big) y(t,x) = 0,\qquad
y(0,x) = 0,
\end{cases}
}
where $\Phi(t,x) = \int_0^1 \pl_y F(t,x, \la y_1 + (1-\la)y_2) d\la$. 
Then, $y(t,x) = 0$  since we can express $y(t,x)$ via the fundamental solution to \rf{linear1}.
\end{proof}
Let us proceed with the global existence.
Define the sequence of stopping times 
\aaa{
\lb{TN}
T_N =  T \we 
\inf \big\{t \in (0,T]: \, \|\eta(t,\fdot)\|_{\C^4_b(\Rnu^n)} >N\big\}, 
}
where $N>0$ is an integer. Note that
since $\eta \in \C^{0,4}_b([0,T]\x \Rnu^n)$ on $\Om_0$, then the 
stopping time $T_N$ is non-zero on $\Om_0$. 
Furthermore, we define
\aaa{
\lb{etaN}
\eta_{N}(t,x) = \eta(t\we T_N,x)  \quad \text{and}  \quad
h_N(x) = h(x)\ind_{\{\|h\|_{\C^2_b(\Rnu^n)} \lt N\}}.
}
Note that for each $\om\in\Om_0$, 
 $\|\eta_{N}\|_{\C^{0,4}_b([0,T]\x \Rnu^n)} \lt N$.
 
 The existence and uniqueness of a global solution
to \rf{viscous} is case $\eta=\eta_N$ is given by Lemma \ref{globalN} below.
\begin{lem}
\lb{globalN}
Let (A1)--(A3) hold. Then, there exists a unique $\mc F_t$-adapted $\C^{0,2}_b$-solution to 
\aaa{
\lb{viscN}
y(t,x) =  h_N(x) + \int_0^t  \big[f(s,x,y) - (y,\nabla)y(s,x)+ \nu\Dl y(s,x)\big] ds +  \eta_N(t,x).
 }
\end{lem}
\begin{proof}
Define $F_{N}(t,x,y)$ by  \rf{rforce} via $\eta_{N}$. Then, 
 $|F_N(t,x,y)|+ |\nab_{(x,y)}F_N(t,x,y)| + |\nab^2_{(x,y)}F_N(t,x,y)| \lt K_N(1+|y|)$,
where $K_N>N$ is a deterministic  constant depending only on $N$.
Consider the backward equation associated to \rf{viscN} by means of substitution 
\rf{substitution} and the time change:
\aaa{
\lb{backward-pdeN}
\bar y(t,x)  =  h_N(x) +  \int_t^T  \big[ \nu \lap \bar y(s,x) - (\bar \eta_{N}(t,x) + \bar y, \nab)\bar y(s,x) + \bar F_{N}(s,x,\bar y) \big] ds.
}
Here 
$\bar F_{N}(t,x,y) = F_{N}(T-t,x,y)$ and $\bar \eta_{N}(t,x) = \eta_{N}(T-t,x)$.

By Theorem \ref{adapted},  on a  deterministic interval $[T-\gm_{K_N}, T]$,
where $\gm_{K_N}$ is the small constant defined by Theorem \ref{adapted},
 there exists an $\mc F_{T-t}$-adapted  $\C^{1,2}_b$-solution $\bar y_N(t,x)$ 
to  equation \rf{backward-pdeN}.
 Then, $y_N(t,x) = \bar y_N(T-t,x) + \eta_N(t,x)$ is an $\mc F_t$-adapted $\C^{0,2}_b$-solution to \rf{viscN} which exists on some set $\Om_N\sub\Om_0$, $\PP(\Om_N) = 1$.
 Remark that for each $\om\in\Om_N$, $\bar y_N(t,x,\om)$ is also a pathwise solution
 to \rf{backward-pdeN}. By Theorem \ref{grad-est}, $\pl_x \bar y_N(t,x,\om)$ is bounded by a constant 
 $\mu_{K_N,T}$ depending only on $K_N$ and $T$ but not depending on the length of the time interval $\gm_{K_N}$.
 Further remark that  $\mu_{K_N,T}$ is the same for all $\om\in\Om_N$.

 Now take $t_1 = \gm_{K_N}$ and consider the equation
 \mmm{
 \lb{visc2}
y(t,x) =  y_N(t_1,x)+ \int_{t_1}^t  \big[f(s,x,y) - (y,\nabla)y(s,x)+ \nu\Dl y(s,x)\big] ds \\ 
+  \eta_N(t,x) - \eta_N(t_1,x).
 }
Note that $\mc F_t = \sg\{B_s, s\in [t_1, t]\} \vee \mc F_{t_1}$ and $y_N(t_1,x)$ is
 $\mc F_{t_1}$-measurable. Further,
by what was proved, $\pl_x y_N(t_1,x)$ is bounded by $\mu_{K_N,T}$.
 Hence, by Theorem
\ref{adapted}, there exists a constant $\gm'_{K_N}$ such that on the time interval $[t_1, t_1 +\gm'_{K_N}]$,
there exists a $\C^{0,2}_b$-solution to \rf{visc2}. Furthermore, for each $t\in [t_1, t_1 +\gm'_{K_N}]$,
this solution is $\mc F_t$-adapted.
In the similar manner, a $\C^{0,2}_b$-solution to \rf{viscN} can be built on the next successive interval
$[t_2,t_2+ \gm'_{K_N}]$, where $t_2 =  \gm_{K_N} + \gm'_{K_N}$. It is important to mention that the initial condition on each short-time interval
has a bounded derivative in $x$ (by the constant $\mu_{K_N,T}$) by Theorem \ref{grad-est}. By glueing the solutions on short-time intervals, we obtain a 
$\C^{0,2}_b$-solution to \rf{viscN} on $[0,T]$.
Remark that this solution is unique by Lemma \ref{lem-unique} since \rf{viscN} can be reduced to equation of type \rf{rforced}
by  substitution \rf{substitution}.
\end{proof}
The main result of this work is Theorem \ref{main} below which
 gives the existence of an $\mc F_t$-adapted $\C^{0,2}_b$-solution to equation \rf{viscous}.
\begin{thm}
\lb{main}
Assume  (A1)--(A3). Then, there exists a unique  $\C^{0,2}_b$-solution to equation  \rf{viscous}
 which is $\mc F_t$-adapted for each $x\in\Rnu^n$.
\end{thm}
 \begin{proof}
Consider equation \rf{viscous}
 for a fixed $\om_0\in\cap_N\Om_N$, 
 where $\Om_N$ is the set of $\om$, where $y_N$ solves \rf{viscN},
 i.e., we regard \rf{viscous} as a deterministic equation. Then, $\eta(t,x,\om_0)$ can be regarded as a bounded
 function in $t$ and $x$. Applying Lemma 
\ref{globalN}, to deterministic equation \rf{viscous}, we obtain the existence and uniqueness of 
a $\C^{0,2}_b$-solution $y(t,x,\om_0)$. Pick an integer $N>0$ such that $\|h(\fdot,\om_0)\|_{\C^2_b(\Rnu^n)} \lt N$. Then,
$h(\fdot,\om_0) = h_N(\fdot,\om_0)$. Further
note that on $[0,T_N(\om_0)]$, equations \rf{viscous} and \rf{viscN} coincide. By Lemma \ref{lem-unique},
$y_N(t,x,\om_0) = y(t,x,\om_0)$ on  $[0,T_N(\om_0)]$. Since $T_N(\om_0) \to T$ as $N\to \infty$, 
then $y_N(t,x,\om_0) \to y(t,x,\om_0)$. This is valid for any $\om_0\in \cap_N\Om_N$. Therefore, 
$y(t,x,\om)$ is $\mc F_t$-adapted.
 \end{proof}

\section{Vanishing viscosity limit} 

\lb{sec5}

Here we investigate the behavior of the solution to \rf{viscous} when the viscosity $\nu$ goes to zero.
Throughout this section, the $\C^{2}_b$-norm of the function $h(x)$ is assumed bounded in $\om$.
At first, we assume that $\eta(t,x)=\eta_N(t,x)$, where $\eta_N(t,x)$ is defined by \rf{etaN}.
 This will allow us to prove 
that the local  vanishing viscosity limit  for equation \rf{rforced}  exists on $[0,\gm_{K_N}]$, where 
$\gm_{K_N}$ is defined in the proof of Lemma \ref{globalN}.

In what follows, $\beta_i$, $i=1,2,\ldots$, denote positive constants, and $\E_\tau$ denote the  conditional expectation with respect to $\mc F_{T-\tau}$. 

\begin{lem}
\lb{lem81}
Assume (A1)--(A3). 
Further assume that $\eta=\eta_N$, and $\|h\|_{\C^2_b}$ is bounded in $\om\in\Om_0$.
Then, for all $\om\in\Om_0$,
the system of forward-backward random equations
\aaa{
\lb{bsde-inviscid}
\begin{cases}
  X^{\tau,x,0}_t = x -\int_{\tau}^t\big(\bar \eta(s,X^{\tau,x,0}_s) + Y^{\tau,x,0}_s \big)ds, \\
Y^{\tau,x,0}_t = h(X^{\tau,x,0}_T) + \int_t^T  \bar F(s, X^{\tau,x,0}_s, Y^{\tau,x,0}_s) \, ds
\end{cases}
}
possesses a unique solution $(X^{\tau,x,0}_t, Y^{\tau,x,0}_t)$ on $[T-\gm_{K_N},T]$ which is continuous in $(\tau,x,t)$.
\end{lem}

\begin{proof}
Forward-backward system \rf{bsde-inviscid} is a particular case of FBSDE \rf{fbsde-new}.
Therefore, if $T-\tau<\gm_{K_N}$, 
then \rf{bsde-inviscid} has a unique solution $(X^{\tau,x,0}_t, Y^{\tau,x,0}_t)$ for each fixed $\om\in\Om_0$.
Further, the uniform boundedness of $Y^{\tau,x,0}_t$ is a direct consequence of the backward equation in \rf{bsde-inviscid}
and Gronwall's inequality. 
Furthermore, \rf{cont7} can be proved for \rf{bsde-inviscid}  pathwise and without involving expectations.
This implies the uniform in $t\in [T-\gm_{K_N},T]$ continuity of the solution $(X^{\tau,x,0}_t, Y^{\tau,x,0}_t)$
in $(\tau,x)$ (as before, it is assumed that $(X^{\tau,x,0}_t, Y^{\tau,x,0}_t)$ is extended to 
$[T-\gm_{K_N}, \tau]$ by $(x, Y^{\tau,x,0}_\tau)$). Therefore, the solution $(X^{\tau,x,0}_t, Y^{\tau,x,0}_t)$
is continuous in $(\tau,x,t)\in [T-\gm_{K_N},T]\x \Rnu^n\x  [T-\gm_{K_N},T]$.
\end{proof}
For each $(t,x,\om) \in [T-\gm_{K_N},T] \x \Rnu^n\x \Om_0$, we define
\aaa{
\lb{y00}
\bar y_0(t,x) = Y^{t,x,0}_t.
}
Let for any viscosity $\nu\in (0,\nu_0]$, where $\nu_0>0$ is a fixed parameter,
$\bar y_\nu(t,x)$ denote the unique $\mc F_{T-t}$-adapted $\C^{1,2}_b$-solution 
to \rf{backward-pde}. 
In the lemma below, we will treat $\nu$ as a ``time'' parameter and
$\bar y_{\fdot\!}:  [0,\nu_0]\x \Om \to \C_b([T-\gm_{K_N},T]\x \Rnu^n)$, $(\nu,\om)\mto \bar y_\nu(\fdot,\fdot)$, as a stochastic
process with values in $\C_b([T-\gm_{K_N},T]\x \Rnu^n)$.
\begin{lem}
\lb{lem999}
Under assumptions of Lemma \ref{lem81},
there exists a constant $\dot\gm_{K_N}<\gm_{K_N}$ such that
there is a continuous version of
\aaa{
\lb{nu-map}
\bar y_{\fdot\!}: [0,\nu_0]\x \Om \to \C_b([T-\dot\gm_{K_N},T]\x \Rnu^n), \quad
(\nu,\om)\mto \bar y_\nu(\fdot,\fdot).
}
\end{lem}
\begin{proof}
Let $(X^{\tau,x,\nu}_t, Y^{\tau,x,\nu}_t, Z^{\tau,x,\nu}_t)$ be the solution to \rf{fbsde-new} associated to $\nu\in (0,\nu_0]$.
As before, sometimes we skip the upper index $(\tau,x)$ (but keep $\nu$).
We have
\aaa{
\lb{fbsde-55}
\begin{cases}
 X^{\nu}_t - X^{\bar \nu}_t = \int_\tau^t \big[ \bar \eta(s,X^{\nu}_s) - \bar \eta(s,X^{\bar \nu}_s) + Y^{\nu}_s -  Y^{\bar \nu}_s\big]ds
 + (\sqrt{2\nu}-\sqrt{2\bar \nu})(W_t-W_\tau),\\
Y^{\nu}_t -  Y^{\bar \nu}_t= h ( X^{\nu}_T) - h ( X^{\bar \nu}_T)
+ \int_t^T  (\bar F_\dl(s, X^{\nu}_s, Y^{\nu}_s) -  \bar F_\dl(s, X^{\bar \nu}_s, Y^{\bar \nu}_s)) \, ds\\
\hspace{9cm} - \int_t^T (Z^{\nu}_s - Z^{\bar \nu}_s) dW_s,
\end{cases}
}
where $\bar F_\dl$ is defined by \rf{gmforce}.
Note that $Z^{\tau,x,0}_t = 0$. By Gronwall's inequality, the forward SDE implies that a.s.
\aaa{
\lb{x0}
 \E_\tau|X^{\nu}_t - X^{\bar \nu}_t|^{2} \lt
\beta_1 \big[ (T-\tau)^{2}\E_\tau |Y^{\nu}_t-  Y^{\bar \nu}_t|^{2} ds + (T-\tau)|\nu-\bar \nu|\big].
}
It\^o's formula applied to the BSDE in \rf{fbsde-55} gives
\mmm{
\lb{ineq1}
\E_\tau|Y^{\nu}_t -  Y^{\bar \nu}_t|^{2}  \lt
\E_\tau |h ( X^{\nu}_T) - h ( X^{\bar \nu}_T)|^{2}\\
+ 2\E_\tau \int_t^T 
(\bar F_\dl(s, X^{\nu}_s, Y^{\nu}_s) -  \bar F_\dl(s, X^{\bar \nu}_s, Y^{\bar \nu}_s), Y^{\nu}_s -  Y^{\bar \nu}_s ) \big] ds \quad \text{a.s.}
} 
From \rf{x0} and \rf{ineq1} it follows that there exists a positive constant $\dot\gm_{K_N}<\gm_{K_N}$ such that 
for each fixed $\nu$ and $\bar\nu$, a.s.,
\aaa{
\lb{y0}
 |y_\nu(\tau,x)-y_{\bar \nu}(\tau,x)|  
\lt \beta_2 |\nu-\bar \nu| \quad \text{for all} \,
x\in\Rnu^n,  \, \tau\in [T-\dot\gm_{K_N},T].
}
Remark that since for each fixed $\nu$ and $\bar\nu$, 
$Y^{\tau,x,\nu}_\tau$ and  $Y^{\tau,x, \bar \nu}$ possess $(\tau,x)$-continuous
modifications, \rf{y0} holds on a set of full $\PP$-measure that does not depend 
on $\tau$ and $x$.
Further remark that the constant $\beta_2$ on the right-hand side of \rf{y0} does not depend on $\tau$ and $x$. 
Therefore, for an integer $p>1$,
\aaa{
\lb{yy}
\E \sup_{x\in\Rnu^n, \tau\in [T-\dot\gm_{K_N},T]}  |y_\nu(\tau,x)-y_{\bar \nu}(\tau,x)|^{2p}
\lt \beta_3|\nu-\bar \nu|^p.
}
By Kolmogorov's continuity theorem (\cite{kunita}, p. 31), there is 
an a.s. $\nu$-continuous version of the stochastic process
$\bar y_{\cdot}\!: [0,\nu_0]\x \Om \to \C_b([T-\dot\gm_{K_N},T]\x \Rnu^n)$,
$(\nu,\om)\mto \bar y_\nu(\fdot,\fdot)$. 
\end{proof}
Lemma \ref{vanvisc} below states the existence of a local vanishing viscosity limit of equation \rf{backward-pde} for $\eta=\eta_N$.
\begin{lem}
\lb{vanvisc}
Let assumptions of Lemma \ref{lem81} be fulfilled.
Then, there exists a positive constant 
$\beta_{K_N}<\dot\gm_{K_N}$ such that
$\bar y_0(t,x)$, defined by \rf{y00}, is a $\C^{1,1}_b$-solution to
 equation \rf{backward-pdeN} with $\nu = 0$
on $[T- \beta_{K_N},T]$. 
Moreover,
as $\nu\to 0$, a.s., $\bar y_\nu(t,x) \to \bar y_0(t,x)$  uniformly in $(x,t)\in\Rnu^n\x [T- \beta_{K_N},T]$,
where  $\bar y_\nu$ is  the $\nu$-continuous version defined by \rf{nu-map}.
\end{lem}
\begin{proof}
Let us prove that for each fixed $x\in\Rnu^n$ and $\tau \in  [T- \beta_{K_N},T]$,
we can take a limit in \rf{backward-pde} as $\nu\to 0$ in the space $L_2(\Om)$,
where $ \beta_{K_N}$ is an appropriate small constant.
Note that the proof
of differentiability of the FBSDE solution (Step 3 of the proof 
 of Theorem \ref{adapted}) holds for the case $\nu=0$ (with $Z^{\tau,x,0}_t=0$).
Therefore, $(X^{\tau,x,0}_t, Y^{\tau,x,0}_t)$ is differentiable in $x$, and 
$(\pl_k X^{\tau,x,0}_t, \pl_k Y^{\tau,x,0}_t, 0)$ satisfies \rf{fbsde-d-short}.
The FBSDE for the triple  $(\pl_k X^{\nu}_t - \pl_k X^{0}_t, \pl_k Y^{\nu}_t - \pl_k Y^{0}_t, \pl_k Z^{\nu}_t)$
takes the form 
\aaa{
\lb{fbsde-nu0}
\begin{cases}
 \pl_k X^{\nu}_t - \pl_k X^{0}_t = - \int_\tau^t\big( \nab \bar \eta(s, X^{0}_s)(\pl_k X^{\nu}_s
  - \pl_k X^{0}_s) + \pl_k Y^{\nu}_s - \pl_k Y^{0}_s+ \xi^X_\nu(s)  \big) ds  \\
\pl_k Y^{\nu}_t - \pl_k Y^{0}_t = \nab h(X^{0}_T)(\pl_k X^{\nu}_T - \pl_k X^{0}_T) 
+ \int_t^T \big[ \nab_2 \bar F_\dl(s, X^{0}_s, Y^{0}_s)(\pl_k X^{\nu}_s - \pl_k X^{0}_s) \\
\hspace{2cm} 
+ \, \nab_3\bar F_\dl(s, X^{0}_s, Y^{0}_s)(\pl_k Y^{\nu}_s - \pl_k Y^{0}_s) + \xi^Y_\nu(s) \big] ds 
+  \int_t^T \pl_k Z^{\nu}_s dW_s
+ \sig^Y_{T,\nu},
\end{cases}
}
where $\xi^X_\nu(s) = -\big(\nab \bar \eta(s,X^{\nu}_s) - \nab \bar\eta(s, X^{0}_s)\big)\pl_k X^{\nu}_s$,
$\sig^Y_{T,\nu} = \big(\nab h(X^{\nu}_T) - \nab h(X^{0}_T)\big)\pl_k X^{\nu}_T$,
\vspace{-3mm}
\mm{
\xi^Y_\nu(s)  = 
 \big(\nab_2 \bar F_\dl(s,X^{\nu}_s, Y^{\nu}_s) - \nab_2 \bar F_\dl(s,X^{0}_s, Y^{0}_s) \big)\pl_k X^{\nu}_s \\
+ \big(\nab_3 \bar F_\dl(s,X^{\nu}_s, Y^{\nu}_s) - \nab_3 \bar F_\dl(s,X^{0}_s, Y^{0}_s) \big)\pl_k Y^{\nu}_s.
}
From \rf{fbsde-nu0}, by standard arguments, we obtain that
there exists a constant
$ \beta_{K_N} < \dot\gm_{K_N}$ such that for all $\tau\in [T-  \beta_{K_N}, T]$,
$x\in\Rnu^n$, and $\nu>0$, 
a.s.,
\aa{
|\pl_k Y^{\tau,x,\nu}_\tau - \pl_k Y^{\tau,x,0}_\tau|^2  \lt 
 \beta_4 \, \E_\tau \big\{  \int_{T-\beta_{K_N}}^T \hsp \big(|\xi^X_\nu(s)|^2 + (|\xi^Y_\nu(s)|^2\big) ds + |\sig^Y_{T,\nu}|^2\big\}.
}
By what was proved, we can choose continuous versions of the maps $[T-\beta_{K_N},T]\x \Rnu^n
\to \C([T-\beta_{K_N},T])$, $(\tau,x)\mto \pl_k X^{\tau,x}$,  $(\tau,x)\mto 
\pl_k Y^{\tau,x}$, $(\tau,x)\mto X^{\tau,x}$, 
$(\tau,x)\mto Y^{\tau,x}$, and of the map $(\tau,x)\mto Y^{\tau,x}_\tau$.
Therefore,
the above  estimate holds on a set of full $\PP$-measure that does
not depend on $\tau$ and $x$.
Hence,  
\mm{
\E \sup_{x\in\Rnu^n, \tau\in [T-\beta_{K_N},T]} |\pl_x \bar y_{\nu}(\tau,x) - \pl_x \bar y_0(\tau,x)|^2 \\ \lt 
\beta_4 \, \E \big\{ \sup_{\tau,x}\E_\tau \int_{T-\beta_{K_N}}^T \hsp \big(|\xi^X_\nu(s)|^2 + (|\xi^Y_\nu(s)|^2\big) ds + |\sig^Y_{T,\nu}|^2\big\}
\to 0 \quad \text{as} \; \nu\to 0
}
by \rf{dif-bounds}, \rf{x0},  and \rf{y0}.
Further, by  \rf{y-bound} and \rf{yd11},
the bounds for $\bar y_\nu(t,x)$ and  $\pl_x \bar y_\nu(t,x)$ 
do not depend on $\nu\in (0,\nu_0]$. 
Therefore, as $\nu\to 0$,  
\mmm{
\lb{non-lin}
 \E \sup_{x\in\Rnu^n, \tau\in [T-\beta_{K_N},T]} |(\bar y_{\nu},\pl_x)\bar y_{\nu}(t,x) - (\bar y_0,\pl_x)\bar y_0(t,x)|^2 \\
\lt  \E \sup_{x\in\Rnu^n, \tau\in [T-\beta_{K_N},T]} \big( |(\bar y_\nu-\bar y_0),\pl_x)\bar y_\nu (t,x)|^2 
+ |(\bar y_0,\pl_x)(\bar y_\nu-\bar y_0)(t,x)|^2 \big)
\to 0.
}
Finally, by \rf{2-bound}, $\lap \bar y_\nu(t,x)$ is bounded uniformly in $\nu \in (0,\nu_0]$ and $(t,x)\in [T-  \beta_{K_N}, T]\x \Rnu^n$. 
This implies that as $\nu\to 0$,
 \aaa{
 \lb{nu-lap}
 \nu\, \E  \sup_{x\in\Rnu^n, \tau\in [T-\beta_{K_N},T]} |\lap \bar y_\nu(t,x)|^2 \to 0.
 }
 Now equation \rf{backward-pde}, together with Lemma \ref{lem999}, \rf{non-lin}, and \rf{nu-lap} imply that, a.s.,
  for all $(t,x)\in [T-  \beta_{K_N}, T] \x \Rnu^n$,
\aaa{
\lb{y90}
\bar y_0(t,x)= h(x) + \int_t^T\big[(\bar y_0,\nab)\bar y_0(s,x) + \bar F_N(s,x,\bar y_0(t,x))\big]\, ds.
}
Further, by Lemma \ref{lem999}, for the $\nu$-continuous version of the process
$\bar y_{\fdot}: [0,\nu_0]\x \Om \to \C([T- \beta_{K_N},T]\x \Rnu^n), \;
(\nu,\om)\mto \bar y_\nu$,
it holds that, a.s.,
\aa{
\sup_{x\in\Rnu^n, \tau\in [T- \beta_{K_N},T]}  |\bar y_\nu(\tau,x)-\bar y_{0}(\tau,x)| \to 0 \quad \text{as} \; \nu\to 0.
}
The lemma is proved. 
\end{proof}

The following theorem is the main result of this section. 
\begin{thm}
Assume (A1)--(A3).
Further, we assume that $\|h\|_{\C^2_b}$ is bounded in $\om\in\Om_0$.
Then, there exists a stopping time $S$, positive a.s.,  such that
 on $[0,S]$
 there exists a $\C^{0,1}_b$-solution $y_0(t,x)$ to the inviscid stochastic Burgers equation 
\aaa{
\label{inviscid}
y(t,x) =  h(x) + \int_0^t  \big[f(s,x,y) - (y,\nabla)y(s,x)\big] ds +  \eta(t,x).
 }
 This solution is $\mc F_t$-adapted for each $x\in\Rnu^n$.
Moreover, if $\td y_0(t,x)$ is another $\C^{0,1}_b$-solution to \rf{inviscid} on $[0, \td S]$, 
where $\td S$ is a positive stopping time,
then, a.s., $\td y_0(t,x) = y_0(t,x)$ on $[0, S\we \td S]$. 
Furthermore, if $y_\nu(t,x)$ is the $\C^{0,2}_b$-solution to \rf{viscous}
(whose existence has been established by Theorem \ref{main}), then there exists a
$\nu$-continuous version of 
$y_{\cdot}\!: [0,\nu_0]\x \Om \to \C_b([0, S]\x \Rnu^n)$,
$(\nu,\om)\mto y_\nu$. In particular, it holds that
 $\lim_{\nu\to 0}y_\nu(t,x) = y_0(t,x)$  a.s., where the limit is uniform in $(x,t)\in\Rnu^n\x [0, S]$.
\end{thm}
\begin{proof}
Let $\bar y_0^N$ be defined by \rf{y00} and associated to a positive integer $N$.  
As it was shown in the proof of Lemma \ref{vanvisc}, $\bar y_0^N$ is a  $C^{1,1}_b$-solution to 
\rf{y90} on $[T-\beta_{K_N},T]$.
Therefore,
$y^N_0(t,x) = \bar y^N_0(T-t,x) + \eta_N(t,x)$  is a $C^{0,1}_b$-solution to 
\aaa{
\label{inviscid1}
y(t,x) =  h(x) + \int_0^t  \big[f(s,x,y) - (y,\nabla)y(s,x)\big] ds +  \eta_N(t,x)
 }
 on $[0,\beta_{K_N}]$. 
Define $S= \beta_{K_N}\we T_N$, where $T_N$ is given by \rf{TN}. By Lemma \ref{lem-unique},
 $y^N_\nu(t,x) = y_\nu(t,x)$ on $[0,S]$  for all $\nu\in (0,\nu_0]$, where $y_\nu(t,x)$ is the unique  $\C^{0,2}_b$-solution to \rf{viscous}.
 Since, by Lemma \ref{vanvisc}, $\lim_{\nu\to 0} y^N_\nu(t,x) = y^N_0(t,x)$, a.s., in the space $\C_b([0, \beta_{K_N}]\x \Rnu^n)$, then
 $y^N_0(t,x) = y_0(t,x)$ on $[0,S]$.
 Thus, we skip the index $N$ when we consider this solution in $[0,S]$.
 Clearly, on $[0,S]$, $y_0(t,x)$ verifies \rf{inviscid} a.s.
 
 Assume, equation \rf{inviscid} has another $\C^{0,1}_b$-solution $\td y_0(t,x)$ which verifies this equation on a random time interval $[0,\td S]$, where
 the stopping time $\td S$ is positive a.s. On $[T-\td S,T]$, we define $\check y_0(t,x) = \td y_0(T-t,x) -\eta(T-t,x)$, 
 and consider equation \rf{eqtdx} below pathwise for each $\tau\in [T-\td S, T]$:
\aaa{
\lb{eqtdx}
\td X^{\tau,x,0}_t = x -\int_{\tau}^t\big(\bar \eta(s,\td X^{\tau,x,0}_s) + \check y_0(s, \td X^{\tau,x,0}_s) \big)ds.
}
Let $\td X^{\tau,x,0}_t$ be the solution to \rf{eqtdx}.
Then, it is straightforward to verify that $\big(\td X^{\tau,x,0}_t, \check y_0(t,\td X^{\tau,x,0}_t)\big)$ is a solution to \rf{bsde-inviscid}.
Indeed, it suffices to note that $\pl_t \check y_0(t,\td X^{\tau,x,0}_t) = (\pl_t \td X^{\tau,x,0}_t,\pl_x)\check y_0(t,\td X^{\tau,x,0}_t)$  
and compute $\pl_t \td X^{\tau,x,0}_t$ via \rf{eqtdx}. By the uniqueness of solution to \rf{bsde-inviscid} on $[T- S\we \td S,T]$,
we conclude that $y_0(t,x) = \td y_0(t,x)$ on $[0, S\we\td S] \x\Rnu^n$ a.s.
The theorem is proved.
\end{proof}

\bibliographystyle{siamplain}

\end{document}